\documentclass{article}
\usepackage{amsmath}
\usepackage{enumitem}            
\usepackage{amsthm}
\usepackage{amssymb}
\usepackage[ansinew]{inputenc}
\usepackage{url}
\usepackage[all,knot,arc,curve,color,frame]{xy}
\input xy
\xyoption{all}
\usepackage{comment}
\usepackage[margin=0.5cm]{geometry}
\usepackage{fullpage,longtable}

\usepackage{soul}               

\theoremstyle{plain}
\newtheorem{lemma}{Lemma}[section]
\newtheorem{theorem}[lemma]{Theorem}
\newtheorem{prop}[lemma]{Proposition}
\newtheorem{cor}[lemma]{Corollary}

\theoremstyle{definition}
\newtheorem{defi}[lemma]{Definition}
\newtheorem{prob}[lemma]{Problem}
\newtheorem{rem}[lemma]{Remark}
\newtheorem{example}[lemma]{Example}
\newtheorem{nota}[lemma]{Notation}

\newcommand\ctr{\sim_{{tr}}}
\newcommand\con{\sim_{{c}}}
\newcommand\cp{\sim_{{p}}}
\newcommand\co{\sim_{{\!o}}}

\newcommand{\inv}{^{-1}}                      
\newcommand{\by}[1]{\overset{\eqref{#1}}=}    
\newcommand{\Epi}{\mathrm{Epi}}

\DeclareMathOperator\dom{dom}
\DeclareMathOperator\ima{im}
\DeclareMathOperator\spa{span}
\DeclareMathOperator\id{id}

\DeclareMathOperator\fend{End}

\DeclareMathOperator\sym{Sym}

\newcommand{\sm}{\setminus}

\newcommand{\al}{\alpha}
\newcommand{\bt}{\beta}
\newcommand{\gam}{\gamma}
\newcommand{\del}{\delta}

\newcommand{\tet}{\theta}
\newcommand{\ups}{\upsilon}
\newcommand{\sig}{\sigma}

\newcommand{\lam}{\lambda}
\newcommand{\ome}{\omega}

\newcommand{\vep}{\varepsilon}
\newcommand{\ale}{\aleph}

\newcommand\pp{\mathbb{P}}
\newcommand\mz{\diamond}
\newcommand\lan{\langle}
\newcommand\ran{\rangle}
\newcommand\join{\bigsqcup}
\newcommand\jo{\sqcup}

\newcommand\gx{\mi^*(X)}

\newcommand\hph{h_{\phi}}
\newcommand\mi{\mathcal{I}}

\newcommand\da{\Delta_\al}
\newcommand\ta{\Theta_\al}
\newcommand\oa{\Omega_\al}
\newcommand\dka{\da^k}
\newcommand\tka{\ta^k}
\newcommand\ua{\Upsilon_{\!\al}}
\newcommand\la{\Lambda_\al}
\newcommand\db{\Delta_\bt}
\newcommand\tb{\Theta_\bt}
\newcommand\ob{\Omega_\bt}

\newcommand\ub{\Upsilon_{\!\bt}}
\newcommand\lb{\Lambda_\bt}
\newcommand\kal{k_\al}
\newcommand\mal{m_\al}
\newcommand\kbt{k_\bt}
\newcommand\mbt{m_\bt}

\newcommand\gd{\mathcal{D}}
\newcommand\gl{\mathcal{L}}
\newcommand\gr{\mathcal{R}}
\newcommand\gh{\mathcal{H}}
\newcommand\gj{\mathcal{J}}
\newcommand\gt{\mathcal{G}}

\newcommand\gll{\,\gl\,}
\newcommand\grr{\,\gr\,}
\newcommand\ghh{\,\gh\,}
\newcommand\gjj{\,\gj\,}
\newcommand{\vea}{\vep_{\!\mbox{\tiny $A$}}}

\title{Four Notions of Conjugacy for Abstract Semigroups}
\author{Jo\~ao Ara\'ujo, Michael Kinyon, Janusz Konieczny, Ant\'{o}nio Malheiro}
\date{}
\begin{document}

\maketitle

\begin{abstract}
The action of any group on itself by conjugation and the corresponding conjugacy
relation play an important role in group theory.
There have been many attempts to find notions of conjugacy in semigroups that
would be useful in special classes of semigroups
occurring in various areas of mathematics, such as semigroups of matrices,
operator and topological semigroups, free semigroups, transition monoids
for automata, semigroups given by presentations with prescribed properties,
monoids of graph endomorphisms, etc. In this paper we study four notions of conjugacy
for semigroups, their interconnections, similarities and dissimilarities. They appeared
originally in various different settings (automata, representation theory, presentations,
and transformation semigroups). Here we study them in full generality.
The paper ends with a large list of open problems.

\smallskip

\noindent\emph{$2010$ Mathematics Subject Classification\/}. 20M07, 20M20,
20M15.

\end{abstract}

\section{Introduction and preliminaries}\label{scon}
\setcounter{equation}{0}

By a notion of conjugacy for a class of semigroups we mean an equivalence
relation defined in the language of that class of semigroups  and coinciding
with the group theory notion of conjugacy whenever the semigroup is a group.
We study three notions of conjugacy in the most general setting
(that is, in the class of all semigroups) and, in view of its importance for representation theory,
we also study one notion that was originally only defined for finite semigroups.

When generalizing a concept, it is sometimes tempting to think that there should be one correct, or even preferred,
generalization. The view we take in this paper is that since semigroup theory is a vast subject, intersecting many areas
of pure and applied mathematics, it is probably not reasonable to expect a
one-size-fits-all notion of conjugacy suitable for all purposes. Searching
for the ``best'' notion of conjugacy is, from our point of view, akin to searching for, say,
the ``best'' topology. Instead, we think that the goal of studying conjugacy in semigroups is to determine
what different notions of conjugacy look like in various classes of semigroups, and how they interact with each other
and with other mathematical concepts.
It is thus incumbent upon individual mathematicians to decide, given their needs, which particular
notion fits best with the class of semigroups under consideration and within the particular context.

In this paper, we consider primarily four notions of conjugacy (and some variations) that we see as especially
interesting given their properties and generality. However, as happens throughout mathematics, stronger  notions
can be obtained by requiring additional properties. Adding to the general requirements in the first paragraph
above, one might require that the notion of conjugacy must be nontrivial, or
first order definable, or that a given set of results about conjugacy in groups carries to some class of
semigroups, etc. Therefore, the years to come will
certainly see the rise of many more systems of equivalence relations for semigroups based on notions of conjugacy.

Before introducing the notions of conjugacy that will occupy us in this paper,
we recall some standard definitions and notation (we generally follow \cite{Ho95}). Other needed
definitions will be given in context.

For a semigroup $S$, we denote by $E(S)$ the set of idempotents of $S$; $S^1$ is
the semigroup $S$ if $S$ is a monoid, or otherwise denotes the monoid
obtained from $S$ by adjoining an identity element $1$. The relation $\leq$
on $E(S)$ defined by $e\leq f$ if $ef=fe=e$ is a partial order on $E(S)$
\cite[p.~69]{Ho95}. A commutative semigroup of idempotents is said to be a
\emph{semilattice}.

An element $a$ of a semigroup $S$ is said to be \emph{regular} if there exists $b\in S$ such that $aba = a$.
Setting $c = bab$, we get $aca = a$ and $cac = c$, so $c$ is an \emph{inverse}
of $a$. Since $a$ is also an inverse of $c$, we often say that $a$ and $c$ are \emph{mutually inverse}.
A semigroup $S$ is \emph{regular} if all elements of $S$ are regular, and it is an
\emph{inverse semigroup} if every element of $S$ has a \emph{unique} inverse.

If $S$ is a semigroup and $a,b\in S$, we say that $a\gll b$ if $S^1a=S^1b$, $a\grr b$ if $aS^1=bS^1$,
and $a\gjj b$ if $S^1aS^1=S^1bS^1$. We define $\gh = \gl\cap \gr$, and $\gd = \gl\lor \gr$, that is,
$\gd$ is the smallest equivalence relation on $S$ containing both $\gl$ and $\gr$.
These five equivalence relations are known as \emph{Green's relations}
\cite[p.\,45]{Ho95}, and are among the most important tools in studying semigroups.

We now introduce the four notions of conjugacy that we will consider in this paper.
As noted, we expect any reasonable notion of semigroup conjugacy to coincide in groups
with the usual notion. For elements $a,b,g$ of a group $G$, if $a=g\inv bg$, then we say that $a$ and
$b$ are \emph{conjugate} and $g$ (or $g\inv$) is a \emph{conjugator} of $a$ and $b$. Conjugacy in
groups has several equivalent formulations that avoid inverses, and hence generalize syntactically to any
semigroup. For example, if $G$ is a group, then $a,b\in G$
satisfy $a=g\inv bg$ (for some $g\in G$) if and only if $a=uv$ and $b=vu$ for some $u,v\in G$
(namely $u=g\inv b$ and $v=g$).
This last formulation has been used to define the following relation on a free semigroup $S$
(see \cite{La79}):
\begin{equation}
\label{econ2}
a\cp b\iff\exists_{u,v\in S^1}\ a=uv\text{ and } b=vu\,.
\end{equation}

If $S$ is a free semigroup, then $\cp$ is an equivalence relation on $S$ \cite[Cor.~5.2]{La79},
and so it can be considered as a notion of conjugacy in $S$. In a general
semigroup $S$, the relation $\cp$ is reflexive and symmetric, but not transitive.
If $a\cp b$ in a semigroup, we say that $a$ and $b$ are \emph{primarily
related} \cite{KuMa09} (hence the subscript in $\cp$). The transitive  closure $\cp^*$ of $\cp$ has
been defined as a conjugacy relation in a general semigroup \cite{Hi06,KuMa07,KuMa09}. Lallement
credited the idea of the relation $\cp$ to Lyndon and Sch\"{u}tzenberger \cite{lyndon}.

Again looking to group conjugacy as a model, for $a,b$ in a group $G$, $a=g\inv bg$ for some $g\in G$
if and only if $ag=gb$ for some $g\in G$ if and only if $bh = ha$, for some $h\in G$ (namely $h=g\inv$).
A corresponding semigroup conjugacy is defined as follows:
\begin{equation}
\label{econ3}
a\co b\iff\exists_{g,h\in S^1}\ ag=gb\text{ and } bh=ha.
\end{equation}
This relation was defined by Otto for monoids presented by finite Thue systems \cite{Ot84},
and, unlike $\cp$,  it is an equivalence relation in any semigroup. However, $\co$ is the universal
relation in any semigroup $S$ with zero. Since it is generally believed \cite{gril,jd,rs} that
$\lim_{n\to \infty} \frac{z_n}{s_n} = 1$, where $s_n$ [$z_n$] is the number of
semigroups [with zero] of order $n$, it would follow that ``almost all'' finite semigroups have a zero
and hence this notion of conjugacy might be of interest only in particular classes of semigroups.

In \cite{AKM14} a new notion of conjugacy was introduced. This notion coincides
with Otto's concept for semigroups without zero, but does not reduce to the universal relation when
$S$ has a zero. The key idea was to restrict the set from which conjugators can be chosen.
For a semigroup $S$ with zero and $a\in S\setminus \{0\}$, let
$\pp(a)$ be the set of all elements $g\in S$ such that $(ma)g\ne0$ for all
$ma\in S^1a\setminus\{0\}$. We also define $\pp(0) = \{0\}$. If $S$ has no zero, we set $\pp(a)=S$
for every $a\in S$. Let $\pp^1(a) = \pp(a)\cup\{1\}$ where $1\in S^1$.
Define a relation $\con$ on any semigroup $S$ by
\begin{equation}
\label{e1dcon}
a\con b\iff\exists_{g\in\pp^1(a)}\exists_{h\in\pp^1(b)}\ ag=gb\textnormal{ and }bh=ha\,.
\end{equation}
(See \cite[\S2]{AKM14} for the motivation of using the sets $\pp^1(a)$.)
Restricting the choice of conjugators, as happens in the definition of $\con$,
is not unprecedented for semigroups. For example, if $S$ is a monoid and $G$ is the group of units
of $S$, we say that $a$ and $b$ in $S$ are \emph{$G$-conjugated} and write $a\sim_G b$
if there there exists $g\in G$ such that $b=g^{-1}ag$ \cite{KuMa07}.
The restrictions proposed in the definition of $\con$ are much less stringent. Their choice was motivated
by considerations in the context of semigroups of transformations. The translation of these considerations
into abstract semigroups resulted in the sets $\pp^1(a)$. (See \cite[\S2]{AKM14} for details.)
Roughly speaking, conjugators selected from $\pp^1(a)$ satisfy the minimal requirements needed to avoid
the pitfalls of~$\co$.

The relation $\con$ turns out to be an equivalence relation on an arbitrary semigroup $S$. Moreover, if
$S$ is a semigroup without zero, then $\con\,\,=\,\,\co$. If $S$ is a free semigroup, then
$\con\,\,=\,\,\co\,\,=\,\,\cp$. In the case where $S$ has a zero, the conjugacy class of $0$
with respect to $\con$ is $\{0\}$.

The last notion of conjugacy that we will consider has been inspired by
considerations in the representation theory of finite semigroups
(for details we refer the reader to Steinberg's book~\cite{Steinberg15}).
Let $M$ be a finite monoid and let $a,b\in M$. We say that $a\ctr b$ if there
exist $g,h\in M$ such that $ghg=g$, $hgh=h$, $hg=a^\omega$,
$gh=b^\omega$, and $ga^{\omega +1}h=b^{\omega +1}$, where, for $a\in M$,
$a^\omega$ denotes the unique idempotent in the monogenic semigroup generated by $a$ (see \cite[\S1.2]{Ho95})
and $a^{\omega + 1} = aa^{\omega}$.  The relation $\ctr$ is an equivalence relation in any finite monoid.

The same notion can be alternatively introduced (see, for example, Kudryavtseva
and Mazorchuk \cite{KuMa09}) via characters of finite-dimensional representations. Given a finite-dimensional complex
representation $\varphi:S\to \fend_{\,\mathbb{C}}(V)$ of a semigroup $S$,
the character of $\varphi$ is the
function $\chi_\varphi: S\to \mathbb{C}$ defined by $\chi_\varphi(s) = \mathrm{trace}(\varphi(s))$
for all $s\in S$. In a finite monoid $S$, $a\ctr b$ if and only if $\chi_\varphi(a)=\chi_\varphi(b)$
(\cite[Thm.~2.2]{McAl1972} or \cite[Prop.~8.9, 8.3 and Thm.~8.10]{Steinberg15})
This explains the subscript notation $\ctr$.

The relation $\ctr$, in its equational definition, can be naturally extended from the class of finite monoids
to the class of epigroups. We need some definitions first. Let $S$ be a semigroup.
An element $a\in S$ is an \emph{epigroup element} (or, more classically, a \emph{group-bound element}) if there exists
a positive integer $n$ such that $a^n$ belongs to a subgroup of $S$,
that is, the $\mathcal{H}$-class $H_{a^n}$ of $a^n$ is a group.
If this positive integer is $1$, then $a$ is said to be \emph{completely regular}. If we denote by $e$ the identity element
of $H_{a^n}$, then $ae$ is in $H_{a^n}$ and we define the \emph{pseudo-inverse} $a'$ of $a$
by $a'=(ae)^{-1}$, where $(ae)^{-1}$ denotes the inverse of $ae$ in the group $H_{a^n}$ \cite[(2.1)]{Shevrin}.
An \emph{epigroup} is a semigroup consisting entirely of epigroup elements, and a
\emph{completely regular semigroup} is a semigroup consisting entirely of
completely regular elements. Finite semigroups and completely regular semigroups are examples of epigroups.
Following Petrich and Reilly \cite{PeRe99} for completely regular semigroups
and Shevrin \cite{Shevrin} for epigroups, it is now customary to view an epigroup $(S,\cdot)$ as a \emph{unary}
semigroup $(S,\cdot,{}')$ where $x\mapsto x'$ is the map sending each element to its pseudo-inverse. In addition, the
${}^\omega$ notation introduced above for finite semigroups can be extended to an epigroup $S$ \cite[\S 2]{Shevrin},
where, for $a\in S$,  $a^\omega$ denotes the idempotent of the group to which some power of $a$ belongs.
(In the finite case, $a^{\omega}$ itself is a power of $a$.)  We can therefore extend
the definition of $\ctr$ from finite monoids to epigroups: for all $a,b$ in 
a epigroup $S$,
\begin{equation}
\label{e1dctr1}
a\ctr b\iff\exists_{g,h \in S^1}\ ghg=g,\, hgh=h,\, ga^{\omega +1}h=b^{\omega +1},\, hg=a^\ome,\mbox{ and } gh=b^\ome.
\end{equation}
In any epigroup, we have $a^\omega=aa'$ (\cite[\S\S2.2.]{Shevrin}), and therefore $a^{\omega+1}=aa'a=a''$.
Thus in epigroups, as is sometimes convenient, we can express the conjugacy relation $\ctr$ entirely in terms of
pseudo-inverses: for all $a,b\in S$,
\begin{equation}
\label{e1dctr2}
a\ctr b\iff\exists_{g,h \in S}\ ghg=g,\, hgh=h,\, ga''h=b'',\, hg=aa',\mbox{ and } gh=bb'.
\end{equation}

We will refer to $\cp$, $\cp^*$, $\co$, $\con$, and $\ctr$ as $p$-conjugacy, $p^*$-conjugacy, $o$-conjugacy,
$c$-conjugacy, and trace conjugacy, respectively. Of course, $\cp$ is a valid notion of conjugacy only in the
class of semigroups in which it is transitive, and trace conjugacy is only defined for epigroups.

For epigroups (and, in particular, for finite semigroups), we have the inclusions depicted in Figure \ref{fig0} 
(which will be justified later).
\begin{figure}[h]
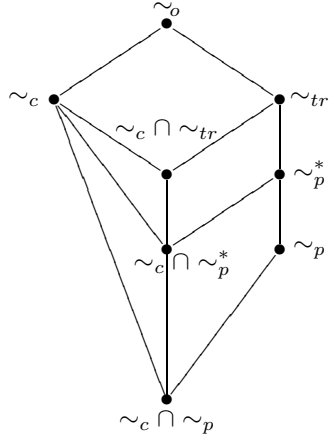

\[
\xy
(-4,40)*{\con};
(0,40)*{\bullet}="c";
(15,52)*{\co};
(15,50)*{\bullet}="o";
(15,36)*{\con\cap\ctr};
(15,30)*{\bullet}="ctr";
(17,18)*{\con\cap \cp^*};
(15,20)*{\bullet}="cp*";
(34,40)*{\ctr};
(30,40)*{\bullet}="tr";
(34,30)*{\cp^*};
(30,30)*{\bullet}="p*";
(34,20)*{\cp};
(30,20)*{\bullet}="p";
(15,-3)*{\con\cap \cp};
(15,0)*{\bullet}="cp";
"o";"c"**\dir{-};
"o";"tr"**\dir{-};
"p*";"tr"**\dir{-};
"p";"p*"**\dir{-};
"cp";"p"**\dir{-};
"cp";"c"**\dir{-};
"ctr";"c"**\dir{-};
"tr";"ctr"**\dir{-};
"cp*";"c"**\dir{-};
"p*";"cp*"**\dir{-};
"ctr";"cp"**\dir{-};
\endxy
\]
\caption{Inclusions between the four conjugacies}\label{fig0}
\end{figure}
The corresponding picture for arbitrary semigroups can be extracted from Figure~\ref{fig0} by removing $\ctr$.
The following semigroup $S$, which is \texttt{SmallSemigroup(7,542155)} of \cite{Smallsemi}, shows that all inclusions
in Figure~\ref{fig0} are strict:
 \[
 \begin{tabular}{c|ccccccc}
 $\cdot$ &0&1&2&3&4&5&6\\ \hline
 0&0&0&0&0&4&4&0\\
 1&0&0&0&0&4&4&0\\
 2&0&0&0&0&4&4&0\\
 3&0&0&0&0&4&4&0\\
 4&4&4&4&4&4&4&4\\
 5&4&4&4&4&4&4&4\\
 6&0&0&2&3&4&5&6
\end{tabular}
\]
Since $S$ has a zero (the element $4$) it follows that $\co\,\,= S\times S$; in addition, it is obvious from the table that
$\cp$ (viewed as a directed graph) consists of all loops together with the edges $0-2$, $0-3$, and $4-5$.
Therefore, the partition induced by $\cp^*$ is
$\{\{0,2,3\},\{4,5\},\{1\},\{6\}\}$. On the other hand, $\ctr$ induces the partition $\{\{0,1,2,3\},\{5,6\},\{4\}\}$.
Finally, we have  $\pp(0)=\pp(1)=\pp(2)=\pp(3)=\pp(6)=\{0,1,2,3,6\}$; $\pp(4)=\{4\}$, and $\pp(5)=\emptyset$. From that
we infer that $\con$ induces the partition $\{\{0,1,2,3,6\},\{4\},\{5\}\}$. Now, $\con\cap \cp$ consists of all loops and the edges $0-2$ and $0-3$;
$\con\cap \cp^*$ induces the partition $\{\{0,2,3\},\{4\},\{5\},\{1\},\{6\}\}$; finally, $\con\cap \ctr$ induces the partition
$\{\{0,1,2,3\},\{4\},\{5\},\{6\}\}$.

In \S\ref{ssym}, we study $c$-conjugacy, trace conjugacy, and $p$-conjugacy in one of the most
important classes of inverse semigroups with proper divisors of zero, namely symmetric inverse semigroups
(see \cite[Thm.~5.1.5]{Ho95}). We give a complete description of
the $c$-conjugacy classes, answering a question posed by the referee of \cite{AKM14}.
In the symmetric inverse semigroup $\mi(X)$ on a set
$X$, we find that $\con\,\,\subset\,\,\cp$ when $X$ is finite, and
$\cp$ and $\con$ are not comparable when $X$ is countably infinite.
Note that $\cp\,\,\subseteq\,\,\co$ in every semigroup $S$
\cite[Thm.~2.2]{AKM14}.
However, as $\mi(X)$ shows, the relation between $\con$ and $\cp$ is more
complex.

In \S\ref{sec:green}, we study the relationship between conjugacies and Green's
relations.
We find that, in general, Green's relations and the conjugacies under
consideration are not
comparable with respect to inclusion,
but there are some comparison results for some transformation semigroups. Our general perception,
however, is that conjugacies and Green's relations form two ``orthogonal'' systems of equivalence relations.

The bulk of our results is contained in \S\ref{sec:epigroups} and \S\ref{scom}. Roughly speaking, in the first we deal with conditions under which the conjugacies tend to be equal; in the second we deal with the opposite situation. Given the definition of  $\ctr$, epigroups form the largest class of semigroups in which all the notions are defined, and hence is the largest class in which all the relations could be equal; therefore  \S\ref{sec:epigroups} only deals with epigroups.
In particular, to have $\cp$ equal to one of the other notions of conjugacy, a necessary condition  is
the transitivity of $\cp$. A complete classification  of the semigroups in which $\cp$ is transitive is still
an open problem. Besides groups and free semigroups \cite[Cor.~5.2]{La79}, a recent result of
Kudryavtseva \cite[Cor.~4]{ganna}
shows that $p$-conjugacy is transitive in completely regular semigroups.
We generalize this result by introducing a wider class of epigroups that
contains completely regular semigroups and
their variants.

In \S\ref{scom}, we prove a number of properties and separation results of the
four notions of conjugacy. 
We conclude the
section by extending various results about
conjugacy in groups to conjugacy in semigroups. For example, if $\sim$ is any
of $\cp$, $\con$,  $\co$ to $\ctr$, then
$a\sim b$ implies $a^k\sim b^k$, just like in groups.

Finally, \S\ref{spro} lists open problems regarding the notions of conjugacy
under discussion, showing how wide open this topic is.

\section{Conjugacy in symmetric inverse semigroups}
\label{ssym}
\setcounter{equation}{0}

The \emph{symmetric inverse semigroup} on a non-empty set $X$ is the semigroup
$\mi(X)$ of partial injective transformations
on $X$ under composition \cite[p.~148]{Ho95}.
The aim of this section is to answer a question of the referee of \cite{AKM14}
regarding $c$-conjugacy in $\mi(X)$ for
a countable $X$, and also compare these results with the existing ones on
the other notions of conjugacy. For $\mi(X)$, with countable $X$,
$p$-conjugacy was described in \cite{GaKo93} (for $X$ finite) and \cite{KuMa07}
(for $X$ countably infinite). It will follow
from these descriptions and our result that in $\mi(X)$,
$\con\,\,\subset\,\,\cp$ if $X$ is finite,
and $\con$ and $\cp$ are not comparable (with respect to inclusion) if $X$ is
countably infinite. We note that since the semigroup
$\mi(X)$ has a zero, $o$-conjugacy in $\mi(X)$ is universal for every $X$.
Also, if $X$ is infinite, then $\mi(X)$ is not an epigroup, so trace conjugacy is only
defined for $\mi(X)$ if $X$ is finite. We will get back to this later.

The importance of symmetric inverse semigroups comes from the fact that every
inverse semigroup can be embedded in $\mi(X)$
for some $X$ \cite[Thm.~5.1.7]{Ho95}.
The role of $\mi(X)$ in the theory of inverse semigroups is analogous to that
of the symmetric group $\sym(X)$ of permutations on $X$ in group theory.

To describe $\con$ in $\mi(X)$, we will use the cycle-chain-ray decomposition
of a partial injective transformation \cite{Ko13}, which is an extension of the
cycle decomposition of a permutation.

We will write functions on the right and compose from left to right; that is,
for $f:A\to B$ and $g:B\to C$, we will
write $xf$, rather than $f(x)$, and $x(fg)$, rather than $g(f(x))$.
Let $\al\in \mi(X)$. We denote the domain of $\al$ by $\dom(\al)$ and the image
of $\al$ by $\ima(\al)$. The union $\dom(\al)\cup\ima(\al)$ will be called the \emph{span}
of $\al$ and denoted $\spa(\al)$.
We say that $\al$ and $\bt$ in $\mi(X)$ are \emph{completely disjoint} if
$\spa(\al)\cap\spa(\bt)=\emptyset$.

\begin{defi}
\label{djoi}
Let $M$ be a set of pairwise completely disjoint elements of $\mi(X)$. The
\emph{join} of the elements of $M$,
denoted $\join_{\gam\in M}\gam$, is the element of $\mi(X)$ whose domain is
$\bigcup_{\gam\in M}\dom(\gam)$
and whose values are defined by
\[
x(\join_{\gam\in M}\gam)=x\gam_0,
\]
where $\gam_0$ is the (unique) element of $M$ such that $x\in\dom(\gam_0)$. If
$M=\emptyset$,
we define $\join_{\gam\in M}\gam$ to be~$0$ (the zero in $\mi(X)$).
If $M=\{\gam_1,\gam_2,\ldots,\gam_k\}$ is finite, we may write the join as
$\gam_1\jo\gam_2\jo\cdots\jo\gam_k$.
\end{defi}

\begin{defi}
\label{dbas}
Let $\ldots,x_{-2},x_{-1},x_0,x_1,x_2,\ldots$ be pairwise distinct elements of
$X$. The following elements of $\mi(X)$
will be called \emph{basic} partial injective transformations on $X$.
\begin{itemize}
  \item A \emph{cycle} of length $k$ ($k\geq1$), written $(x_0\,x_1\ldots\,
x_{k-1})$,
is an element $\del\in\mi(X)$ with $\dom(\del)=\{x_0,x_1,\ldots,x_{k-1}\}$,
$x_i\del=x_{i+1}$ for all $0\leq i<k-1$,
and $x_{k-1}\del=x_0$.
  \item A \emph{chain} of length $k$ ($k\geq1$), written $[x_0\,x_1\ldots\,
x_k]$,
is an element $\tet\in\mi(X)$ with $\dom(\tet)=\{x_0,\ldots,x_{k-1}\}$ and
$x_i\tet=x_{i+1}$ for all $0\leq i\leq k-1$.
  \item A \emph{double ray}, written $\lan\ldots\,x_{-1}\,x_0\,x_1\ldots\ran$,
is an element $\ome\in\mi(X)$ with
$\dom(\omega)=\{\ldots,x_{-1},x_0,x_1,\ldots\}$ and $x_i\ome=x_{i+1}$ for all
$i$.
  \item A \emph{right ray}, written $[x_0\,x_1\,x_2\ldots\ran$,
is an element $\ups\in\mi(X)$ with $\dom(\ups)=\{x_0,x_1,x_2,\ldots\}$ and
$x_i\ups=x_{i+1}$ for all $i\geq0$.
  \item A \emph{left ray}, written $\lan\ldots\, x_2\,x_1\,x_0]$,
is an element $\lam\in\mi(X)$ with $\dom(\lam)=\{x_1,x_2,x_3,\ldots\}$ and
$x_i\lam=x_{i-1}$ for all $i>0$.
\end{itemize}
By a \emph{ray} we will mean a double, right, or left ray.
\end{defi}

We note the following.
\begin{itemize}
  \item The span of a basic partial injective transformation is exhibited by
the
notation. For example, the span
of the right ray $[1\,2\,3\ldots\ran$ is $\{1,2,3,\ldots\}$.
  \item The left bracket in ``$\eta=[x\ldots$'' indicates that
$x\notin\ima(\eta)$; while the right bracket
in ``$\eta=\ldots\,x]$'' indicates that $x\notin\dom(\eta)$. For example, for
the chain $\tet=[1\,2\,3\,4]$,
$\dom(\tet)=\{1,2,3\}$ and $\ima(\tet)=\{2,3,4\}$.
  \item A cycle $(x_0\,x_1\ldots\, x_{k-1})$ differs from the corresponding
cycle in the symmetric group
of permutations on $X$ in that the former is undefined for every $x\in
X\sm\{x_0,x_1,\ldots,x_{k-1}\}$, while the latter
fixes every such $x$.
\end{itemize}

The following decomposition result was proved in \cite[Prop.~2.4]{Ko13}.

\begin{prop}\label{pdec}
Let $\al\in\mi(X)$ with $\al\ne0$. Then there exist unique sets: $\da$ of
cycles,
$\ta$ of chains, $\oa$ of double rays, $\ua$ of right rays, and $\la$ of left
rays
such that
the transformations in $\da\cup\ta\cup\oa\cup\ua\cup\la$ are pairwise
completely
disjoint and
\begin{equation}
\al=\join_{\del\in\da}\!\!\del\jo\join_{\tet\in\ta}\!\!\tet\jo\join_{\ome\in\oa}
\!\!\ome\jo\join_{\ups\in\ua}\!\!\ups
\jo\join_{\lam\in\la}\!\!\lam.\label{edec}
\end{equation}
\end{prop}

We will call the join (\ref{edec}) the \emph{cycle-chain-ray decomposition} of
$\al$. If
$\eta\in\da\cup\ta\cup\oa\cup\ua\cup\la$, we will say that $\eta$ is
\emph{contained} in $\al$
(or that $\al$ \emph{contains} $\eta$).
We note the following.
\begin{itemize}
  \item If $\al\in\sym(X)$, then
$\al=\join_{\del\in\da}\!\del\jo\join_{\ome\in\oa}\!\ome$
(since $\ta=\ua=\la=\emptyset$), which corresponds to the usual cycle
decomposition
of a permutation \cite[1.3.4]{Sc64}.
  \item If $\dom(\al)=X$, then
$\al=\join_{\del\in\da}\!\del\jo\join_{\ome\in\oa}\!\ome\jo\join_{\ups\in\ua}
\!\ups$
(since $\ta=\la=\emptyset$), which corresponds to the decomposition given in
\cite{Le91}.
  \item If $X$ is finite, then
$\al=\join_{\del\in\da}\!\del\jo\join_{\tet\in\ta}\!\tet$
(since $\oa=\ua=\la=\emptyset$), which is the decomposition given in
\cite[Theorem~3.2]{Li96}.
\end{itemize}

For example, if $X=\{1,2,3,4,5,6,7,8,9\}$, then
\[
\al=\begin{pmatrix}1&2&3&4&5&6&7&8&9\\3&6&-&5&9&8&-&2&-\end{pmatrix}\in\mi(X)
\]
written in cycle-chain decomposition (no rays since $X$ is finite) is
$\al=(2\,6\,8)\jo[1\,3]\jo[4\,5\,9]$. The following $\bt$ is an example of an
element of $\mi(\mathbb Z)$ written
in cycle-chain-ray decomposition:
\[
\bt=(2\ 4)\jo[6\ 8\ {10}]\jo\lan\ldots-6\,-4\,-2\,-1\,-3\,-5\,\ldots\ran\jo[1\ 
5\ 
9\ {13}\,\ldots\ran\jo\lan\ldots {15}\ {11}\ 7\ 3].
\]

\begin{nota}\label{nzero}
We will fix an element $\mz\notin X$. For $\al\in\mi(X)$ and $x\in X$, we will
write $x\al=\mz$
if and only if $x\notin\dom(\al)$. We will also assume that $\mz\al=\mz$. With
this notation, it will make sense
to write $x\al=y\bt$ or $x\al\ne y\bt$ ($\al,\bt\in\mi(X)$, $x,y\in X$) even
when $x\notin\dom(\al)$ or $y\notin\dom(\bt)$.
\end{nota}

\begin{nota}\label{ndka}
For $0\ne\al\in\mi(X)$, let $\da$ be the set of cycles and $\ta$ be the set of
chains that occur in the
cycle-chain-ray decomposition of $\al$ (see (\ref{edec})). For $k\geq1$, we
denote by $\dka$ the set of cycles in $\da$ of length~$k$,
and by $\tka$ the set of chains in $\ta$ of length $k$.
\end{nota}

\begin{defi}\label{dtyp}
Let $\al\in\mi(X)$. The sequence of cardinalities
\[
\lan|\da^1|,|\da^2|,|\da^3|,\ldots;|\ta^1|,|\ta^2|,|\ta^3|,\ldots;|\oa|,|\ua|,
|\la|\ran
\]
(indexed by the elements of the ordinal $2\ome+3$)
will be called the \emph{cycle-chain-ray type} of $\al$.
This notion generalizes the cycle type of a permutation \cite[p.~126]{DuFo04}.
Suppose $\dom(\al)$ is finite.
Then $\al$ does not have any rays and its cycle-chain-ray type reduces to the
\emph{cycle-chain type}
\[
\lan|\da^1|,|\da^2|,|\da^3|,\ldots;|\ta^1|,|\ta^2|,|\ta^3|,\ldots\ran.
\]
\end{defi}

The cycle-chain-ray type of $\al$ is completely determined by the \emph{form}
of
the cycle-chain-ray
decomposition of $\al$. The form is obtained from the decomposition by omitting
each occurrence of the symbol ``$\jo$''
and replacing each element of $X$ by some generic symbol, say ``$*$.'' For
example,
$\al=(2\,6\,8)\jo[1\,3]\jo[4\,5\,9]$ has the form $(*\,*\,*)[*\,*][*\,*\,*]$,
and
\[
\bt=(2\ 4)\jo[6\ 8\ 10]\jo\lan\ldots-6\,-4\,-2\,-1\,-3\,-5\,\ldots\ran\jo[1\ 5\
9\ 13\,\ldots\ran\jo\lan\ldots15\ 11\ 7\ 3]
\]
has the form
$(*\,*)[*\,*\,*]\lan\ldots *\,*\,*\,\ldots\ran[*\,*\,*\,\ldots\ran\lan\ldots
*\,*\,\,*]$.

A \emph{directed graph} (or a \emph{digraph}) is a pair $\Gamma=(A,R)$ where
$A$
is a set (not necessarily finite and possibly empty)
and $R$ is a binary relation on $A$. Any element $x\in A$ is called a
\emph{vertex} of $\Gamma$,
and any pair $(x,y)\in R$ is called an \emph{arc} of $\Gamma$. We will call a
vertex $y$ \emph{terminal}
if there is no $x\in A$ such that $(x,y)\in R$.

Let $\Gamma_1=(A_1, R_1)$ and $\Gamma_2=(A_2, R_2)$ be digraphs.
A mapping $\phi:A_1\to A_2$
is called a \emph{homomorphism} from $\Gamma_1$ to $\Gamma_2$ if for all
$x,y\in
A_1$,
if $(x,y)\in R_1$, then $(x\phi,y\phi)\in R_2$ \cite{HeNe04}.

\begin{defi}\label{drh}
Let $\Gamma_1=(A_1, R_1)$ and $\Gamma_2=(A_2, R_2)$ be digraphs. A homomorphism
$\phi:A_1\to A_2$
is called a \emph{restrictive homomorphism} (or an \emph{r-homomorphism}) from
$\Gamma_1$ to $\Gamma_2$ if
for every terminal vertex $x$ of $\Gamma_1$, $x\phi$ is a terminal vertex of
$\Gamma_2$.
\end{defi}

Any partial transformation $\al$ on a set $X$ (injective or not) can be
represented by the digraph
$\Gamma(\al)=(A_\al,R_\al)$, where $A_\al=\spa(\al)$
and for all $x,y\in A_\al$, $(x,y)\in R_\al$ if and only if $x\in\dom(\al)$ and
$x\al=y$.

The following proposition is a special case of \cite[Thm.~3.8]{AKM14}.

\begin{prop}
\label{pjal1}
For all $\al,\bt\in\mi(X)$, $\al\con\bt$ if and only if there are
$\phi,\psi\in\mi(X)$
such that $\phi$ is an r-homomorphism from $\Gamma(\al)$ to $\Gamma(\bt)$
and $\psi$ is an r-homomorphism from $\Gamma(\bt)$ to $\Gamma(\al)$.
\end{prop}

\begin{defi}\label{dtas}
Let $\ldots,x_{-1},x_0,x_1,\ldots$ be pairwise distinct elements of $X$.
Let $\del=(x_0\ldots x_{k-1})$, $\tet=[x_0\,x_1\,\ldots\,x_k]$,
$\ome=\lan\ldots x_{-1}\, x_0\, x_1\ldots\ran$, $\ups=[x_0\, x_1\,
x_2\ldots\ran$, and
$\lam=\lan\ldots x_2\,x_1\,x_0]$. For any $\eta\in\{\del,\tet,\ome,\ups,\lam\}$
and any $\phi\in\mi(X)$ such that $\spa(\eta)\subseteq\dom(\phi)$, we define
$\eta\phi^*$ to be $\eta$
in which each $x_i$ has been replaced with $x_i\phi$. For example,
\[
\del\phi^*=(x_0\phi\,\,x_1\phi\ldots x_{k-1}\phi)\mbox{ and
}\lam\phi^*=\lan\ldots x_2\phi\,\,x_1\phi\,\,x_0\phi].
\]
Consider $\tet=[x_0\,x_1\ldots\, x_k]$,
$\ome=\lan\ldots\,x_{-1}\,x_0\,x_1\ldots\ran$, $\ups=[x_0\,x_1\,x_2\ldots\ran$,
and $\lam=\lan\ldots\, x_2\,x_1\,x_0]$ in $\mi(X)$. Then
any $[x_i\,x_{i+1}\ldots\, x_k]$ ($0\leq i<k$) is a \emph{terminal segment} of
$\tet$;
any $[x_i\,x_{i+1}\,x_{i+2}\ldots\ran$ is a terminal segment of $\ome$;
any $[x_i\,x_{i+1}\,x_{i+2}\ldots\ran$ ($i\geq0$) is a terminal segment of
$\ups$;
and any $[x_i\,x_{i-1}\ldots\, x_0]$ ($i\geq1$) is a terminal segment of $\lam$.
\end{defi}

The following proposition follows easily from more general results proved in
\cite{AKM14}
(see \cite[Prop.~4.18 and Prop.~7.3]{AKM14}).

\begin{prop}\label{pjal2}
Let $\al,\bt,\phi\in\mi(X)$. Then $\phi$ is an $r$-homomorphism from
$\Gamma(\al)$ to $\Gamma(\bt)$ if and only if for all
$k\geq1$, $\del\in\da^k$,
$\tet\in\ta^k$, $\ome\in\oa$, $\ups\in\ua$, and $\lam\in\la$:
\begin{enumerate}[label=\textup{(\arabic*)}]
  \item $\del\phi^*\in\db^k$, $\ome\phi^*\in\ob$, and
$\lam\phi^*\in\lb$;
  \item either there is a unique $\tet_1\in\tb^m$ with $m\geq k$
such that $\tet\phi^*$ is a terminal segment of $\tet_1$
or there is a unique $\lam_1\in\lb$ such that $\tet\phi^*$ is a terminal
segment
of $\lam_1$;
  \item either there is a unique $\ups_1\in\ub$ such that
$\ups\phi^*$ is a terminal segment of $\ups_1$
or there is a unique $\ome_1\in\ob$ such that $\ups\phi^*$ is a terminal
segment
of $\ome_1$.
\end{enumerate}
\end{prop}

\begin{defi}
\label{djal3}
Let $\al,\bt,\phi\in\mi(X)$ such that $\phi$ is an $r$-homomorphism from
$\Gamma(\al)$ to $\Gamma(\bt)$.
We define a mapping
$\hph:\da\cup\ta\cup\oa\cup\ua\cup\la\to\db\cup\tb\cup\ob\cup\ub\cup\lb$ by:
\[
\eta\hph =
\begin{cases}
\eta\phi^* &\text{ if }\eta\in\da\cup\oa\cup\la\,, \\
\tet_1 & \text{ if }\eta\in\ta\text{ and }\eta\phi^*\text{ is a terminal
segment
of }\tet_1\in\tb\,,\\
\lam_1 & \text{ if }\eta\in\ta\text{ and }\eta\phi^*\text{ is a terminal
segment
of }\lam_1\in\lb\,,\\
\ups_1 & \text{ if }\eta\in\ua\text{ and }\eta\phi^*\text{ is a terminal
segment
of }\ups_1\in\ub\,,\\
\ome_1 & \text{ if }\eta\in\ua\text{ and }\eta\phi^*\text{ is a terminal
segment
of }\ome_1\in\ob\,.
\end{cases}
\]
Note that $\hph$ is well defined (by Proposition~\ref{pjal2}) and injective
(since $\phi$ is injective).
\end{defi}

For a countable set $X$, we define two cardinal numbers that will be crucial in
our characterization of $c$-conjugacy in
the semigroup $\mi(X)$. We denote by $\mathbb Z_+$ the set of positive integers
and
by $\mathbb{N}$ the set $\mathbb{Z}_+\cup\{0\}$.

\begin{defi}
\label{dkm}
Let $X$ be countable and suppose $\al\in\mi(X)$.
We define $\kal\in\mathbb N\cup\{\ale_0\}$ by
\[
\kal=\sup\{k\in\mathbb Z_+:\ta^k\ne\emptyset\}.
\]
If $\ta^k=\emptyset$ for every $k\in\mathbb Z_+$, we define $\kal$ to be $0$.

Suppose $\kal\in\mathbb Z_+$, that is, $\kal$ is the largest positive integer
$k$ such that $\ta^k\ne\emptyset$.
We define $\mal\in\mathbb N$ by
\[
\mal=\max\{m\in\{1,2,\ldots,\kal\}:|\ta^m|=\ale_0\}.
\]
If $\ta^m$ is finite for every $m\in\{1,2,\ldots,\kal\}$, we define $\mal$ to
be
$0$.
\end{defi}

For any chain $\tet$ in $\mi(X)$, we denote the length of $\tet$ by $l(\tet)$.
For example, if $\tet=[1\,2\,3]$ then $l(\tet)=2$.

\begin{lemma}\label{lkal}
Let $X$ be countably infinite and let $\al,\bt\in\mi(X)$.
Suppose that $\kal=\kbt=\ale_0$. Then there exists an injective mapping
$p:\ta\to\tb$
such that for every $\tet\in\ta$, if $\tet\in\ta^k$ and $\tet p\in\tb^m$, then
$m\geq k$.
\end{lemma}
\begin{proof}
Since $\kbt=\ale_0$, the set $\{k\in\mathbb Z_+:\tb^k\ne\emptyset\}$ is
unbounded, which implies that
there is a sequence $\eta_1,\eta_2,\eta_3,\ldots$ of chains in $\tb$ such that
$l(\eta_1)<l(\eta_2)<l(\eta_3)<\ldots$.
Since $\kal=\ale_0$, $\ta$ is countably infinite. Let
$\ta=\{\tet_1,\tet_2,\tet_3,\ldots\}$. For every $i\in\mathbb Z_+$,
select $n_i\in\mathbb Z_+$ such that
$l(\tet_i)\leq l(\eta_{n_i})$. Then $p:\ta\to\tb$ defined by
$\tet_i p=\eta_{n_i}$
is a desired injective mapping.
\end{proof}

\begin{theorem}\label{tcha}
Suppose that $X$ is countable.
Let $\al,\bt\in\mi(X)$. Then $\al\con\bt$ if and only if the following
conditions are satisfied:
\begin{enumerate}[label=\textup{(\arabic*)}]
  \item $|\da^k|=|\db^k|$ for every $k\in\mathbb Z_+$,
$|\oa|=|\ob|$, and $|\la|=|\lb|$;
  \item if $\oa$ is finite, then $|\ua|=|\ub|$; and
  \item if $\la$ is finite, then
\begin{enumerate}[label=\textup{(\roman*)}]
  \item $\kal=\kbt$; and
  \item if $\kal\in\mathbb Z_+$, then $\mal=\mbt$ and for every
$k\in\{\mal+1,\ldots,\kal\}$, $|\ta^k|=|\tb^k|$.
\end{enumerate}
\end{enumerate}
\end{theorem}
\begin{proof}
Suppose $\al\con\bt$. By Proposition~\ref{pjal1}, there exists
$\phi\in\mi(X)$ such that $\phi$ is an $r$-homomorphism from $\Gamma(\al)$ to
$\Gamma(\bt)$. Let $k\in\mathbb Z_+$.
Define $f_k:\da^k\to\db^k$ by $\del f_k=\del\hph$, $g:\oa\to\ob$ by $\ome
g=\ome\hph$, and
$d:\la\to\lb$ by $\lam d=\lam\hph$.
Each of the mappings $f_k$, $g$, and $d$ is injective since
$\hph$ is injective. Thus
$|\da^k|\leq|\db^k|$, $|\oa|\leq|\ob|$, and $|\la|\leq|\lb|$.
By symmetry, $|\db^k|\leq|\da^k|$, $|\ob|\leq|\oa|$, and $|\lb|\leq|\la|$.
Hence (1) holds.

Suppose $\oa$ is finite. Then $g:\oa\to\ob$ defined above is a bijection (since
$g$ is injective
and $|\oa|=|\ob|$). Thus for every $\ome_1\in\ob$, there is $\ome\in\oa$
such that $\ome\hph=\ome g=\ome_1$. Since $\hph$ is injective, it follows that
for every $\ups\in\ua$, $\ups\hph\in\ub$ (since $v\hph$ can not belong to $\ob$),
which implies $|\ua|\leq|\ub|$. By symmetry, $|\ub|\leq|\ua|$. Hence (2) holds.

Suppose $\la$ is finite. Then, by the foregoing argument for $\oa$ and $\ua$
applied to $\la$ and $\ta$,
we conclude that $|\ta|=|\tb|$ and that for every $\tet\in\ta$,
$\tet\hph\in\tb$.
Suppose to the contrary that $\kal\ne\kbt$.
We may assume that $\kal>\kbt$. Then there exists $k\in\mathbb Z_+$ such that
$\kbt<k\leq\kal$ and $\ta^k\ne\emptyset$.
Select some $\tet\in\ta^k$. Then $\tet\hph$ is a terminal segment of some
$\tet_1\in\tb$. But this is a contradiction
since $k>\kbt$ and $\tb^m=\emptyset$ for every
$m>\kbt$. Thus $\kal=\kbt$.

Let $\kal\in\mathbb Z_+$. Suppose to the contrary that $\mal\ne\mbt$.
We may assume that $\mal>\mbt$.
By definition, $|\ta^{\mal}|=\ale_0$.
For every $\tet\in\ta^{\mal}$, $\tet\hph$ is a terminal segment of some
$\tet_1\in\tb$, so $\tet\hph\in\tb^l$ for some $l$ with
$\kbt\geq l\geq \mal>\mbt$. But this is a contradiction since
$\hph$ is injective, the set $\{\tet\hph:\tet\in\ta^{\mal}\}$
is infinite,
and the set $\tb^{\mal}\cup\ldots\cup\tb^{\kbt}$ is finite.
Thus $\mal=\mbt$.

Finally, suppose to the contrary that there exists $k\in\{\mal+1,\ldots,\kal\}$
such that $|\ta^k|\ne|\tb^k|$.
Select the largest such $k$. We may assume that $|\ta^k|>|\tb^k|$. Then
$|\ta^k\cup\ldots\cup\ta^{\kal}|>|\tb^k\cup\ldots\cup\tb^{\kal}|$
and $\hph$ maps $\ta^k\cup\ldots\cup\ta^{\kal}$ to
$\tb^k\cup\ldots\cup\tb^{\kal}$, which is a contradiction since
$\hph$ is injective. Hence $|\ta^k|=|\tb^k|$ for every
$k\in\{\mal+1,\ldots,\kal\}$.
We have proved (3), which concludes the direct part of the proof.

Conversely, suppose that conditions (1), (2) and (3) are satisfied.
We will define an injective homomorphism $\phi$ from $\Gamma(\al)$ to
$\Gamma(\bt)$. By (1), for every $k\in\mathbb Z_+$,
there is an injective mapping $f_k:\da^k\to\db^k$.

Suppose that both $\oa$ and $\la$ are infinite. Then
$|\oa\cup\ua|=|\ob|$ and
$|\la\cup\ta|=|\lb|$,
and so there are injective mappings $g:\oa\cup\ua\to\ob$ and
$d:\la\cup\ta\to\lb$.
For all $k\geq1$, $\del\in\da^k$, $\ome\in\oa$, $\lam\in\la$, $\ups\in\ua$, and
$\tet\in\ta^k$,
we define $\phi$ on
$\spa(\del)\cup\spa(\ome)\cup\spa(\lam)\cup\spa(\ups)\cup\spa(\tet)$
in such a way that $\del\phi^*=\del f_k$, $\ome\phi^*=\ome g$, $\lam\phi^*=\lam d$,
$\ups\phi^*$ is a terminal segment of $\ups g$, and $\tet\phi^*$ is a terminal
segment of $\tet d$.
Note that this defines $\phi$ for every vertex $x$ in $\Gamma(\al)$. By the
definition of $\phi$ and Proposition~\ref{pjal2},
$\phi\in\mi(X)$ and $\phi$ is an $r$-homomorphism from $\Gamma(\al)$ to
$\Gamma(\bt)$.

Suppose that $\oa$ is finite and $\la$ is infinite. Then $|\ua|=|\ub|$ by (2),
and so there exists an injective mapping
$j:\ua\to\ub$. Let $f_k:\da^k\to\db^k$ ($k\in\mathbb Z_+$) and
$d:\la\cup\ta\to\lb$ be the injective mappings defined in
the previous paragraph.
Since $|\oa|=|\ob|$, there exists an injective mapping $g:\oa\to\ob$.
We define $\phi$ as in the previous paragraph, except that
$\ups\phi^*=\ups j$ for every $\ups\in\ua$. Again,
$\phi\in\mi(X)$ and $\phi$ is an $r$-homomorphism from $\Gamma(\al)$ to
$\Gamma(\bt)$.

Suppose that $\oa$ is infinite and $\la$ is finite. Then $\kal=\kbt$ by (3)(i).
Let $f_k:\da^k\to\db^k$ ($k\in\mathbb Z_+$) and $g:\oa\cup\ua\to\ob$ be the
injective mappings defined in
the case when both $\oa$ and $\la$ are infinite.
Since $|\la|=|\lb|$, there exists an injective mapping
$d:\la\to\lb$.

Suppose that $\kal=\ale_0$. Then by Lemma~\ref{lkal}, there is an injective
mapping $p:\ta\to\tb$
such that for every $\tet\in\ta$, if $\tet\in\ta^k$ and $\tet p\in\tb^m$, then
$m\geq k$.
We define $\phi$ as in the case when both $\oa$ and $\la$ are infinite, except
that
$\tet\phi^*$ is a terminal segment of $\tet p$ for every $\tet\in\ta$. Again,
$\phi\in\mi(X)$ and $\phi$ is an $r$-homomorphism from $\Gamma(\al)$ to
$\Gamma(\bt)$.

Suppose that $\kal<\ale_0$. If $\kal=0$ then $\ta=\tb=\emptyset$. Suppose that
$\kal\in\mathbb Z_+$. Then by (3)(ii),
$\mal=\mbt$ and for every $k\in\{\mal+1,\ldots,\kal\}$, $|\ta^k|=|\tb^k|$. Let
$m=\mal$.
We have $|\ta^1\cup\ldots\cup\ta^m|=|\tb^m|=\ale_0$ and $|\ta^k|=|\tb^k|$ for
every $k>m$.
Thus, there are injective mappings $s:\ta^1\cup\ldots\cup\ta^m\to\tb^m$ and
$t_k:\ta^k\to\tb^k$ for every $k>m$.
We define $\phi$ (whether $\kal$ is $0$ or not) as in the case when both $\oa$
and $\la$ are infinite, except that
for every $\tet\in\ta$,
$\tet\phi^*$ is a terminal segment of $\tet s$ if $\tet\in\ta^k$ with $1\leq
k\leq m$,
and $\tet\phi^*$ is a terminal segment of $\tet t_k$ if $\tet\in\ta^k$ with
$k>m$. As in the previous cases,
$\phi\in\mi(X)$ and $\phi$ is an $r$-homomorphism from $\Gamma(\al)$ to
$\Gamma(\bt)$.

Finally, if both $\oa$ and $\la$ are finite, we define an injective
$r$-homomorphism $\phi$ from $\Gamma(\al)$ to $\Gamma(\bt)$
as in the case when $\oa$ is infinite and $\la$ is finite, except that
$\ups\phi^*=\ups j$ for every $\ups\in\ua$, where
$j:\ua\to\ub$ is an injective mapping from the case when $\oa$ is finite and
$\la$ is infinite.

We have proved that there exists an injective $r$-homomorphism $\phi$ from
$\Gamma(\al)$ to $\Gamma(\bt)$.
By symmetry, there exists an injective $r$-homomorphism $\psi$ from
$\Gamma(\bt)$ to $\Gamma(\al)$.
Hence $\al\con\bt$ by Proposition~\ref{pjal1}.
\end{proof}

Suppose that $X$ is finite. Then for every $\al\in\mi(X)$,
$\oa=\ua=\la=\emptyset$, $\kal\ne\ale_0$,
and $\mal=0$ if $\kal\in\mathbb Z_+$. Thus Theorem~\ref{tcha} implies the
following corollary,
which generalizes the result for the symmetric group $\sym(X)$
\cite[Proposition~11, page~126]{DuFo04}.

\begin{cor}\label{ccha}
Suppose that $X$ is finite. Then for all
$\al,\bt\in\mi(X)$, $\al\con\bt$ if and only if $\al$ and $\bt$ have the same
cycle-chain type.
\end{cor}

\begin{rem} \label{remarkSymInvFiniteOtto}
 By Corollary~\ref{ccha}, for a finite set $X$, the relation $\con$ on
$\mi(X)$ can also be
characterized by: $\al\con\bt$ if and only if there exists a permutation
$\sigma$ on the set $X$ such that   $\al= \sigma^{-1}\bt \sigma$.
\end{rem}

Corollary~\ref{ccha} implies that if $X$ is finite, then in $\mi(X)$, $\con$ is
strictly included in $\cp$.

\begin{prop}
\label{pcpf}
Suppose that $X$ is finite with $|X|\geq 2$. Then $\con\,\,\subset\,\,\cp$ in
$\mi(X)$.
\end{prop}
\begin{proof}
Let $\al,\bt\in\mi(X)$ and suppose that $\al\con\bt$. By
Remark~\ref{remarkSymInvFiniteOtto},
there exists
$\sig\in\sym(X)$
such that $\sig\inv \al\sig=\bt$. For $\mu=\al\sig$ and $\nu=\sig\inv$ in
$\mi(X)$,
we have $\mu\nu=\al$ and $\nu\mu=\bt$, and so $\al\cp\bt$.

We have proved that $\con\,\,\subseteq\,\,\cp$.
The inclusion is strict. Select $x,y\in X$ with $x\ne y$. Then for $\al=[x\,y]$
and $\bt=0$ in $\mi(X)$,
$\al\cp\bt$ (since $\al=\al(y)$ and $\bt=(y)\al$) but $(\al,\bt)\notin\,\,\con$
by Corollary~\ref{ccha}.
\end{proof}

Since $\cp\,\,\subseteq\,\,\cp^*$ in any semigroup, we also have
$\con\,\,\subset\,\,\cp^*$ in $\mi(X)$ when $X$ is finite.
The relation $\cp^*$ in a finite $\mi(X)$ was characterized by Ganyushkin and
Kormysheva
\cite{GaKo93} (see also \cite[Thm.~1]{KuMa07}): for all $\al,\bt\in\mi(X)$,
$\al\cp^*\bt$ if and only if $\al$ and $\bt$ have the same cycle type (while
there are no restrictions
on the chain type of $\al$ and $\bt$).

Regarding $\ctr$ in $\mi(X)$, for a finite $X$, we have $\al\ctr\bt$ if and only if
$\al$ and $\bt$ have the same cycle type \cite[Ex. 8.4]{Steinberg15}.
Therefore, in these semigroups, $\ctr\,\,=\,\,\cp^*$. Thus, in $\mi(X)$ and  for
finite $X$, we have the following chain:
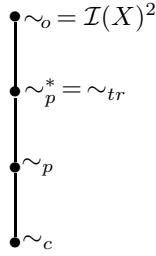
\[
\xy
(3,10)*{\con};
(0,10)*{\bullet}="c&p";
(3,20)*{\cp};
(0,20)*{\bullet}="p";
(8,30)*{\cp^*\,=\,\ctr};
(0,30)*{\bullet}="p*";
(10,40)*{\co\,=\mi(X)^2};
(0,40)*{\bullet}="o";
"o";"p*"**\dir{-};
"p*";"p"**\dir{-};
"p";"c&p"**\dir{-};
\endxy
\]

Proposition~\ref{pcpf} does not extend to the infinite case. Suppose that $X$
is
countably infinite.
Consider the following transformations in $\mi(X)$:
\begin{align*}
\al&=[y_0\,y_1\,y_3]\jo\lan\ldots\,x^1_2\,x^1_1\,x^1_0]\jo\lan\ldots\,x^2_2\,
x^2_1\,x^2_0]\jo\lan\ldots\,x^3_2\,x^3_1\,x^3_0]\jo\ldots,\\
\bt&=\lan\ldots\,z^1_2\,z^1_1\,z^1_0]\jo\lan\ldots\,z^2_2\,z^2_1\,z^2_0]
\jo\lan\ldots\,z^3_2\,z^3_1\,z^3_0]\jo\ldots.
\end{align*}
Then $\da=\db=\oa=\ob=\ua=\ub=\emptyset$ and $\la=\lb=\ale_0$.
Thus $\al\con\bt$ by Theorem~\ref{tcha}. By \cite[Lem.~4]{KuMa07}, if $\al$ and
$\bt$ were $p$-conjugate, then there would exist
an injective mapping $j:\ta^2\to\tb^1\cup\tb^2\cup\tb^3$. Since
$\ta^2=\{[y_0\,y_1\,y_2]\}$ and $\tb^1\cup\tb^2\cup\tb^3=\emptyset$, such a
mapping does not exist, and so
$(\al,\bt)\notin\,\,\cp$.

Now consider $\al=[y_0\,y_1\,y_2]$ and $\bt=[z_0\,z_1]$ in $\mi(X)$. Then
$\al\cp\bt$ by \cite[Lemma~4]{KuMa07}, but
$\al$ and $\bt$ are not $c$-conjugate by Theorem~\ref{tcha}
(since $\la=\emptyset$, $\kal=2$, and $\kbt=1$). Thus $(\al,\bt)\notin\,\,\con$.

The foregoing examples prove the following proposition.

\begin{prop}\label{pcpi}
Suppose that $X$ is countably infinite. Then, with respect to inclusion, $\cp$
and $\con$ are not comparable in $\mi(X)$.
\end{prop}

Since $\cp^*$ is the transitive closure of $\cp$ and $\con$ is an equivalence
relation, it follows
from Proposition~\ref{pcpi} that if $X$ is infinitely countable, then
$\cp^*$ and $\con$ are not comparable in $\mi(X)$ either.
For a countably infinite set $X$, the relation $\cp^*$ in $\mi(X)$ was
characterized by Kudryavtseva and Mazorchuk
\cite[Thm.~2]{KuMa07}.

Therefore, in $\mi(X)$, for a countably infinite $X$, we have the following
diamond:
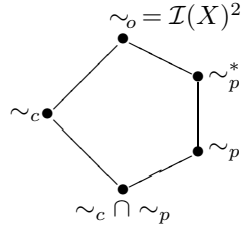
\[
\xy
(-3,10)*{\con};
(0,10)*{\bullet}="c";
(10,-3)*{\con\cap \cp};
(10,0)*{\bullet}="c&p";
(23.5,5)*{\cp};
(20,5)*{\bullet}="p";
(23.5,15)*{\cp^*};
(20,15)*{\bullet}="p*";
 (17,23)*{\co\,=\mi(X)^2};
(10,20)*{\bullet}="o";
"o";"c"**\dir{-};
"o";"p*"**\dir{-};
"p*";"p"**\dir{-};
"p";"c&p"**\dir{-};
"c";"c&p"**\dir{-};
\endxy
\]

If $X$ is infinite, the semigroup $\mi(X)$ is not an epigroup, and hence $\ctr$ is not defined in $\mi(X)$.
However, in \S\ref{sec:epigroups}, we show that $\ctr$ can be defined, and is an equivalence relation, on the set
of epigroup elements of an arbitrary semigroup. We then characterize $\ctr$ as the relation on the set of epigroup elements
of $\mi(X)$ for a countably infinite $X$ (Theorem~\ref{ttrix}).

\section{Conjugacy and Green's relations}
\label{sec:green}
\setcounter{equation}{0}

Green's relations play an important role in studying semigroups. In a
group, any two elements are $\gt$-related, for any Green
relation $\gt$. Thus any two conjugate group elements are $\gt$-related. The
general situation for semigroups is quite different. In this section, we will
show that Green's relations and our four conjugacies are not comparable
in
general,
but there are some inclusion results for the symmetric inverse semigroup
$\mi(X)$ and its subsemigroup
consisting of full injective transformations on $X$.

Fixing some terminology, for a set $X$ and $\al:X\to X$,
the \emph{kernel} of $\al$ is the equivalence relation on $X$ defined by
$\ker(\al)=\{(x,y)\in X\times X: x\al=y\al\}$.

\begin{theorem}\label{tgra}
Let $\gt$ be any Green relation and let $\sim\,\,\in\{\cp,\cp^*,\ctr,\con,\co\}$. Then
there exists a semigroup $S$ such that
$\gt\,\not\subseteq\,\,\sim$ and $\sim\,\,\not\subseteq\,\gt$ in $S$.
\end{theorem}
\begin{proof}
Suppose that {$\sim\ \in\{\cp,\cp^*,\ctr\}$} and consider $S=\mi(X)$, where $X=\{1,2\}$.
In any $\mi(X)$, we have $\al\gjj\bt\iff|\dom(\al)|=|\dom(\bt)|$ and
$\al\ghh\bt\iff(\dom(\al)=\dom(\bt)\mbox{ and }\ima(\al)=\ima(\bt))$.
In any semigroup, $\gj$ is the largest and $\gh$ is the smallest Green relation
with respect to inclusion.
Let $\al=[1\,2]$ and $\bt=0$ in $\mi(X)$. Then $\al\cp\bt$ since $\al=\al(2)$
and $\bt=(2)\al$, but
$(\al,\bt)\notin\,\,\gj$ since $|\dom(\al)|=1$ and $|\dom(\bt)|=0$. Hence
$\cp\,\,\not\subseteq\,\gj$, and so $\cp\,\,\not\subseteq\,\gt$.
{It follows that $\cp^*,\ctr\,\,\not\subseteq\,\gt$
since $\cp\,\,\subseteq\,\,\cp^*\,\,\subseteq\,\,\ctr$ in any finite semigroup (see Figure~\ref{fig0}).}
Now let $\gam=(1)\jo(2)=\id_X$ and $\del=(1\,2)$ in $\mi(X)$. Then
$\gam\ghh\del$, but {$(\gam,\del)\notin\,\ctr$ since,
by \cite[Ex.~8.4]{Steinberg15}, for $X$ finite, $\gam\ctr\del$ in $\mi(X)$ if and only if
$\gam$ and $\del$ have the same cycle type.}
Hence $\gh\,\not\subseteq\,\,{\ctr}$, and so
$\gt\,\not\subseteq\,\,{\ctr}$.
{It follows that $\gt\,\not\subseteq\,\,\cp,\cp^*$
since $\cp\,\,\subseteq\,\,\cp^*\,\,\subseteq\,\,\ctr$.}

Suppose that $\sim\,\,=\,\,\con$ and consider $S=T(X)$, where $X=\{1,2,3\}$ and
$T(X)$ is the semigroup of all
full transformations on $X$.
In any $T(X)$, we have $\al\gjj\bt\iff|\ima(\al)|=|\ima(\bt)|$ and
$\al\ghh\bt\iff(\ker(\al)=\ker(\bt)\mbox{ and }\ima(\al)=\ima(\bt))$.
Let $\al=\begin{pmatrix}1&2&3\\3&3&3\end{pmatrix}$ and
$\bt=\begin{pmatrix}1&2&3\\2&3&3\end{pmatrix}$ in $T(X)$.
Then $\al\con\bt$ by \cite[Cor.~6.3]{AKM14}, but
$(\al,\bt)\notin\,\,\gj$ since $|\ima(\al)|=1$ and $|\ima(\bt)|=2$. Hence
$\con\,\,\not\subseteq\gj$, and so $\con\,\,\not\subseteq\gt$.
Now let $\gam=(1)\jo(2)\jo(3)=\id_X$ and $\del=(1\,2\,3)$ in $T(X)$. Then
$\gam\ghh\del$, but $(\gam,\del)\notin\,\con$
by \cite[Cor.~6.3]{AKM14}. Hence $\gh\,\not\subseteq\,\,\con$, and so
$\gt\,\not\subseteq\,\,\con$.
Since $T(X)$ does not have a zero, we have $\con\,\,=\,\,\co$ in $T(X)$. Thus
the foregoing argument can be applied to $\co$,
which concludes the proof.
\end{proof}

Although $c$-conjugacy is not comparable with Green's relations in general, it
is strictly included in Green's relation $\gj$
in the symmetric inverse semigroup on a countable set.

\begin{prop}
\label{pgrb}
Suppose that $X$ is countable with $|X|\geq2$. Then $\con\,\,\subset\,\gj$ in
$\mi(X)$.
\end{prop}
\begin{proof}
Let $\al,\bt\in\mi(X)$ with $\al\con\bt$. Suppose that $\dom(\al)$ is infinite.
Then $\dom(\bt)$ is also
infinite by Theorem~\ref{tcha}. Thus $|\dom(\al)|=|\dom(\bt)|=\ale_0$, which
implies $\al\gjj\bt$.
Suppose that $\dom(\al)$ is finite.
Then, by Theorem~\ref{tcha}, $\al$ and $\bt$
have the same cycle-chain decomposition, which implies
$|\dom(\al)|=|\dom(\bt)|$. Thus $\al\gjj\bt$ in this case also.
We have proved that $\con\,\,\subseteq\,\gj$. The inclusion is strict since for
$x,y\in X$ with $x\ne y$,
$\al=(x)\jo(y)$ and $\bt=(x\,y)$ in $\mi (X)$ are $\gj$-related but not
$c$-conjugate.
\end{proof}

By the proof of Theorem~\ref{tgra}, $\cp\,\,\not\subseteq\,\gj$ in $\mi(X)$
when
$|X|\geq2$.
However, $\cp$ is strictly included in $\gj$ in the semigroup of \emph{full}
injective transformations
on a countably infinite set $X$.

Denote by $\gx$ the subsemigroup of $\mi(X)$ consisting of all transformations
$\al\in\mi(X)$ with $\dom(\al)=X$.
If $X$ is finite, then $\gx=\sym(X)$ but this is not the case for an infinite
$X$.
The semigroup $\gx$ is universal for right cancellative semigroups with no
idempotents
(except possibly the identity), that is, any such semigroup can be embedded in
$\gx$ for some $X$ \cite[Lemma~1.0]{ClPr64}.

If $\al\in\gx$, then there are no chains or left rays in the cycle-chain-ray
decomposition of $\al$, that is,
$\ta=\la=\emptyset$. By \cite[Thm.~2.3]{Ko10}, for all $\al,\bt\in\gx$,
$\al\gjj\bt$ if and only if $|X\setminus\ima(\al)|=|X\setminus\ima(\bt)|$.
For every $\al\in\gx$, the set $X\setminus\ima(\al)$ consists of the initial
points of the right rays on $\al$,
so $|X\setminus\ima(\al)|=|\ua|$. Thus, for all $\al,\bt\in\gx$,
\begin{equation}
\label{egra}
  \al\gjj\bt\mbox{ in $\gx$}\iff|\ua|=|\ub|.
\end{equation}
\begin{lemma}\label{lgra}
For all $\al,\bt\in\gx$, $\al\cp\bt$ in $\gx$ if and only if $\al\cp\bt$ in
$\mi(X)$.
\end{lemma}
\begin{proof}
Let $\al,\bt\in\gx$. If $\al\cp\bt$ in $\gx$, then $\al\cp\bt$ in $\mi(X)$
since
$\gx\subseteq\mi(X)$. Conversely, suppose
that $\al\cp\bt$ in $\mi(X)$. Then $\al=\mu\nu$ and $\bt=\nu\mu$ for some
$\mu,\nu\in\mi(X)$.
Since $\dom(\al)=X$ and $\al=\mu\nu$, we have $\dom(\mu)=X$. Similarly,
$\dom(\nu)=X$. Thus $\mu,\nu\in\gx$,
and so $\al\cp\bt$ in $\gx$.
\end{proof}

Let $\al,\bt\in\gx$, where $X$ is countably infinite. By
\cite[Lem.~4]{KuMa07}, $\al\cp\bt$ in $\mi(X)$ if and only if $|\da^k|=|\db^k|$
for all $k\in\mathbb Z_+$,
$|\oa|=|\ob|$, and $|\ua|=|\ub|$. Thus, by Lemma~\ref{lgra}, for all
$\al,\bt\in\gx$,
\begin{equation}
\label{egrb}
  \al\cp\bt\mbox{ in $\gx$}\iff\mbox{$\forall_{k\in\mathbb
Z_+}|\da^k(\al)|=|\db^k(\bt)|$, $|\oa|=|\ob|$, and $|\ua|=|\ub|$}.
\end{equation}
For c-conjugacy, we have the following results for an arbitrary set $X$
\cite[Thm.~7.6]{AKM14}:
\begin{equation}
\label{egrc}
  \al\con\bt\mbox{ in $\gx$}\iff\mbox{$\forall_{k\in\mathbb
Z_+}|\da^k(\al)|=|\db^k(\bt)|$, $|\oa|=|\ob|$,
and $|\ua|+|\oa|=|\ub|+|\ob|$}.
\end{equation}
Now, when $X$ is countably infinite, $p$-conjugacy is strictly included in
$\gj$ in $\gx$. In fact, we have
an even stronger result.

\begin{theorem}
\label{tgrd}
Suppose that $X$ is countably infinite. Then $\cp\,\,=\,\,\con\,\cap\,\,\gj$ in
$\gx$. Moreover, $\cp\,\,\subset\,\,\con$ and $\cp\,\,\subset\,\gj$.
\end{theorem}
\begin{proof}
The equality $\cp\,\,=\,\,\con\,\cap\,\,\gj$ follows immediately from
\eqref{egra}, \eqref{egrb}, and \eqref{egrc}. Thus
$\cp\,\,\subseteq\,\,\con$ and $\cp\,\,\subseteq\,\gj$.
Let $X=\{x^i_j:i,j\in\mathbb Z_+\mbox{ with $i\geq1$}\}\cup\{y_j:j\in\mathbb
Z_+\}$.
Consider
\begin{align*}
\al&=[y_0\,y_{-1}\,y_1\,y_{-2}\,y_2\ldots\ran\jo\lan\ldots\,x^1_{-1}\,x^1_0\,
x^1_1\ldots\ran\jo
\lan\ldots\,x^2_{-1}\,x^2_0\,x^2_1\ldots\ran\jo\lan\ldots\,x^3_{-1}\,x^3_0\,
x^3_1\ldots\ran\jo\ldots,\\
\bt&=\lan\ldots\,y_{-1}\,y_0\,y_1\ldots\ran\jo\lan\ldots\,x^1_{-1}\,x^1_0\,
x^1_1\ldots\ran\jo
\lan\ldots\,x^2_{-1}\,x^2_0\,x^2_1\ldots\ran\jo\lan\ldots\,x^3_{-1}\,x^3_0\,
x^3_1\ldots\ran\jo\ldots
\end{align*}
in $\gx$.
Then $\da=\db=\emptyset$, $|\oa|=|\ob|=\ale_0$, $|\ua|=1$, and $|\ub|=0$. Thus
$\al\con\bt$ by \eqref{egrc},
but $(\al,\bt)\notin\,\cp$ by \eqref{egrb}. Hence $\cp\,\,\subset\,\,\con$.
Now, let $X=\{x,y\}\cup\{z_1,z_2,z_3,\ldots\}$ and consider
\[
\gam=(x\,y)\jo[z_1\,z_1\,z_3\ldots\ran\,\mbox{ and
}\,\del=(x)\jo(y)\jo[z_1\,z_1\,z_3\ldots\ran
\]
in $\gx$. Then $\gam\gjj\del$ by \eqref{egra},
but $(\gam,\del)\notin\,\cp$ by \eqref{egrb}. Hence $\cp\,\,\subset\,\gj$.
\end{proof}

Transformations $\al$ and $\bt$ from the proof of Theorem~\ref{tgrd} are
$c$-conjugate but not $\gj$-related.
Thus in $\gx$, where $|X|=\ale_0$, $\con$ is not included in $\gj$. However,
the following result holds for an arbitrary infinite set $X$.

\begin{prop}
\label{pgrc}
Suppose that $X$ is infinite. Let $\al,\bt\in\gx$ be transformations such that
$\al$ has finitely many double chains.
If $\al\con\bt$ then $\al\gjj\bt$.
\end{prop}
\begin{proof}
Suppose that $\al\con\bt$. Then $|\oa|=|\ob|$
and $|\ua|+|\oa|=|\ub|+|\ob|$ by \eqref{egrb}. Since $|\oa|$ is finite, it
follows that $|\ua|=|\ub|$, and so
$\al\gjj\bt$ by \eqref{egra}.
\end{proof}

Since the semigroup $\gx$ does not have a zero, $\con\,\,=\,\,\co$ in $\gx$, so
Theorem~\ref{tgrd} and Proposition~\ref{pgrc}
also hold for $o$-conjugacy. The symmetric inverse semigroup $\mi(X)$ does have
a zero,
so $o$-conjugacy is the universal relation in any $\mi(X)$.
Since $\con$ and $\gj$ are equivalence relations in any semigroup,
it follows from Theorem~\ref{tgrd} that $\cp$ is transitive in $\mi^*(X)$ for a countably infinite $X$.
Thus Theorem~\ref{tgrd} also holds for $\cp^*$. Trace conjugacy is not defined in $\mi(X)$ or $\mi^*(X)$
when $X$ is infinite.

\section{Conjugacy in epigroups and epigroup elements}
\label{sec:epigroups}
\setcounter{equation}{0}
The principal aim of this section is to explore the relations between the four conjugacies
in epigroups, the largest class for which all four notions can be defined.  We will prove that in any epigroup,
$\cp\ \subseteq\ \cp^*\ \subseteq\ \ctr\ \subseteq\ \co$ (see Figure~\ref{fig0}).
We will also investigate when and which conjugacies coincide in a variety of epigroups that contains all variants
of completely regular semigroups.
For background information on epigroups, we refer the reader to the survey paper of Shevrin \cite{Shevrin}.

Let $S$ be a semigroup. As noted in the introduction, an element $a\in S$ is an \emph{epigroup element}
(or a \emph{group-bound element}) if there exists a positive integer $n$ such that $a^n$ is
contained in a subgroup of $S$. The smallest $n$ for which this is satisfied is the \emph{index} of $a$,
and for all $k\geq n$, $a^k$ is contained in the group $\gh$-class of $a^n$.
Let $\Epi(S)$ denote the set of all epigroup elements of $S$ and let $\Epi_n(S)$ denote the subset of
$\Epi(S)$ consisting of elements of index no more than $n$. Thus $\Epi_m(S)\subseteq \Epi_n(S)$ for $m\leq n$ and
$\Epi(S) = \bigcup_{n\geq 1} \Epi_n(S)$. The elements of $\Epi_1(S)$ are more commonly called
\emph{completely regular} (or \emph{group elements}).

For $a\in \Epi_n(S)$, the maximum subgroup of $S$ containing $a^n$ is its $\mathcal{H}$-class $H$.
Let $e$ denote the identity element of $H$. Then $ae = ea$ is in $H$ and we define the \emph{pseudo-inverse}
$a'$ of $a$ by $a'=(ae)^{-1}$, the inverse of $ae$ in the group $H$ \cite[(2.1)]{Shevrin}. This
leads to a characterization: $a\in \Epi(S)$ if and only if there exists a positive
integer $n$ and a (necessarily unique) $a'\in S$ such that the following hold (\cite[Section 2]{Shevrin}):
\begin{equation}\label{etfh}
a'aa' = a'\,,\quad aa'=a'a\,,\quad a^{n+1} a' = a^n\,.
\end{equation}
If $a$ is an epigroup element, then so is $a'$ with $a'' = aa'a$. The element $a''$ is
always completely regular and $a''' = a'$. Borrowing finite semigroups standard notation (\cite{rs,Steinberg15}),
for an epigroup element $a$, we set $a^{\omega} = aa'$. We also have  
$a^\omega=a''a'=a'a''$,   $(a')^{\omega} = (a'')^{\omega} = a^{\omega}$, and 
more generally $a^\omega = (aa')^m=(a')^m a^m = a^m(a')^m$, for all $m>0$. 

A semigroup $S$ is said to be an epigroup if $\Epi(S) = S$. If $\Epi_1(S) = S$ (that is, if $S$ is
a union of groups), then $S$ is called a \emph{completely regular} semigroup. For $n>0$, the class
$\mathcal{E}_n$ consists of all epigroups $S$ such that $S = \Epi_n(S)$; thus $\mathcal{E}_1$ is
the class of completely regular semigroups.

The conclusion of the following lemma is an identity in epigroups, but here we need a version for
epigroup elements. The lemma seems to be a folk result, but we include a brief proof for completeness.

\begin{lemma}
\label{lem:xyyx}
Let $S$ be a semigroup and suppose $xy,yx\in \Epi(S)$ for some $x,y\in S$. Then $(xy)'x=x(yx)'$.
\end{lemma}
\begin{proof}
Let $n$ denote the larger of the indices of $xy$ and $yx$. Then
\begin{align*}
(xy)^{\omega} x &= ((xy)')^{n+1}(xy)^{n+1} x = ((xy)')^{n+1} x (yx)^{n+1}  = ((xy)')^{n+1} x (yx)^n (yx)' \\
&= ((xy)')^{n+1} (xy)^n x (yx)' = (xy)'x(yx)'\,.
\end{align*}
By a dual calculation, we also have $x(yx)^{\omega} = (xy)'x(yx)'$, and thus
\begin{equation}
\label{eqn:xyxy_tmp}
(xy)^{\omega}x = x(yx)^{\omega}\,.
\end{equation}
Now we compute
\[
(xy)'x = (xy)'(xy)^{\omega}x \by{eqn:xyxy_tmp} (xy)'x(yx)^{\omega} = (xy)'xyx(yx)'
= (xy)^{\omega} x (yx)' \by{eqn:xyxy_tmp} x(yx)^{\omega} (yx)' = x(yx)'\,,
\]
as claimed.
\end{proof}

Throughout the rest of the section, the condition $gh = a^{\omega}$, $hg = b^{\omega}$ for some $a,b\in \Epi(S)$,
some $g,h\in S^1$, will recur frequently (as, for example, in the definition of $\ctr$). We record two obvious
consequences of this for later use:
\begin{equation}
\label{eqn:ghgh}
a^{\omega} g = g b^{\omega} \qquad\text{and}\qquad b^{\omega} h = h a^{\omega}\,.
\end{equation}
Indeed, both sides of the first equation are equal to $ghg$ and both sides of the second are equal to $hgh$.

The relation $\ctr$ is not, in general, well-defined for an arbitrary semigroup $S$, but it is a
well-defined relation on $\Epi(S)$: for $a,b\in \Epi(S)$, we set
\begin{equation}
\label{ectr3}
a\ctr b \iff \exists_{g,h\in S^1}\ ghg=g,\ hgh=h,\ ha''g=b'',\ gh = a^{\omega},\ hg = b^{\omega}.
\end{equation}
In fact many of the results on $\ctr$ do not require the whole semigroup to be an epigroup, rather only the involved elements must be epigroup elements;  as an illustration, the next  eight results will be proved  on $\ctr$ restricted to epigroup elements.

We start by observing that the asymmetry in our definition of $\ctr$, which 
follows \cite{Steinberg15},
is only for the sake of brevity.

\begin{lemma}
\label{lem:tr_sym}
Let $S$ be a semigroup, let $a,b\in \Epi(S)$, and suppose there exist $g,h\in S^1$ such that
$gh = a^{\omega}$ and $hg = b^{\omega}$. The following are equivalent.
\begin{align*}
  \textup{(1)} &\quad ha''g = b'';   & \textup{(2)} &\quad gb''h = a'';
& \textup{(3)} &\quad a'' g = g b''; & \textup{(4)} &\quad b'' h = h a''; \\
  \textup{(5)} &\quad ha'g = b';     & \textup{(6)} &\quad gb'h = a';
& \textup{(7)} &\quad a' g = g b';   & \textup{(8)} &\quad b' h = h a'\,.
\end{align*}
\end{lemma}
\begin{proof}
(1)$\implies$(2): $gb''h = gha''gh = a^{\omega} a'' a^{\omega} = a''$.

(2)$\implies$(3): $a''g = gb''hg = gb''b^{\omega} = gb''$.

(3)$\implies$(1): $ha''g = hgb'' = b^{\omega}b'' = b''$.

(1)$\implies$(4)$\implies$(2) follows by an obvious symmetry.

To get (5)$\implies$(6)$\implies$(7)$\implies$(5) and (5)$\implies$(8)$\implies$(6),
we just repeat the same calculations with $a'$ in place of $a$ and $b'$ in place of $b$. Here we use
$a''' = a'$, $b''' = b'$, $(a')^{\omega} = a^{\omega}$ and $(b')^{\omega} = b^{\omega}$.

Showing $(3)\iff(7)$ will conclude the proof. Assume (3). Then
\[
a' g = a' a^{\omega} g \by{eqn:ghgh} a' g b^{\omega} = a' g b'' b' = a' a'' g b' = a^{\omega} g b' \by{eqn:ghgh} g b^{\omega} b' = g b'\,.
\]
This establishes (7). Conversely, if (7) holds, then since $a''' = a'$, $b''' = b'$, we may repeat the same calculation,
replacing $a$ with $a'$ and $b$ with $b'$ to get (3).
\end{proof}

\begin{prop}
\label{pctrainv}
Let $S$ be a semigroup and let $a,b\in \Epi(S)$. Then $a\ctr b$ if and only if $a'\ctr b'$.
\end{prop}
\begin{proof}
This follows from Lemma~\ref{lem:tr_sym} together with $a''' = a'$, $b''' = b'$, 
{$a^{\omega}=(a')^{\omega}$, and $b^{\ome}=(b')^{\omega}$.}
\end{proof}

One theme of this section is to discuss when various notions of conjugacy coincide.
The following lemma will be useful later when we discuss epigroups in which all notions
on the right side of Figure~\ref{fig0} coincide. Although we will not use it right
away, we state it here because it is a lemma about epigroup elements (in fact, idempotents)
in arbitrary semigroups.

\begin{lemma}
\label{lem:epi-ef}
Let $S$ be a semigroup. Suppose $e,f\in E(S)$ satisfy $e\leq f$ and $e\ctr f$. Then $e = f$.
\end{lemma}
\begin{proof}
Since $e\ctr f$, there exist $g,h\in S^1$ such that $ghg=g$, $hgh=h$, $gh=e$, $hg=f$ and
$heg = f$ (using $e'' = e^{\omega} = e$ and $f'' = f^{\omega} = f$). We have $he = h(gh) =
(hg)h = fh$, and so $e = fe = f(hg) = (fh)g = heg = f$.
\end{proof}

We now provide alternative definitions of $\ctr$ and compare trace conjugacy to $p$-conjugacy.
In particular, we show that the requirement that $g$ and $h$ be mutually inverses
can be omitted from the definition of $\ctr$ (see \eqref{ectr3}).

\begin{theorem}
\label{thm:trace}
Let $S$ be a semigroup. For $a,b\in \Epi(S)$, the following are equivalent:
\begin{enumerate}[label=\textup{(\arabic*)}]
\item $a\ctr b$;
\item $\exists_{g,h\in S^1}\ ha''g=b'',\ gh = a^{\omega},\ hg = b^{\omega}$
\item $\exists_{g, h\in S^1}\ a''g = gb'',\ gh = a^{\omega},\ hg = b^{\omega}$;
\item $\exists_{g, h\in S^1}\ ag = gb,\ bh = ha,\ gh = a^{\omega},\ hg = b^{\omega}$;
\item $\exists_{g, h\in S^1}\ hgh=h,\ ha''g = b'',\ gb''h = a''$;
\item $a'' \cp b''$.
\end{enumerate}
\end{theorem}
(The asymmetries in the statements of the theorem are explained by Lemma~\ref{lem:tr_sym}.)
\begin{proof}
The implication  (1)$\implies$(2) is trivial. Assume (2)
and set $\bar{g} = a^{\omega} g$ and $\bar{h} = b^{\omega} h$. Then
\begin{align*}
 \bar{h}a''\bar{g} &=b^\ome h a'' a^{\omega} g = b^\ome h a'' g = b^{\omega} b'' 
= b''\,,\\
 \bar{g}\bar{h} &= a^{\omega} g b^{\omega} h \by{eqn:ghgh} a^{\omega} gh a^{\omega} = a^{\omega} a^\ome a^{\omega}=a^{\omega}\,, \\
 \bar{h}\bar{g} &= b^{\omega} h a^{\omega} g \by{eqn:ghgh} b^{\omega} hg b^{\omega} = b^{\omega} b^\ome b^{\omega} =b^{\omega}\,, \\
 \bar{g}\bar{h}\bar{g} &= a^{\omega} a^{\omega} g = a^{\omega} g = \bar{g}\,,\\
 \bar{h}\bar{g}\bar{h} &= b^{\omega} b^{\omega} h = b^{\omega} h = \bar{h}\,.
\end{align*}
This proves (1). The equivalence (2)$\iff$(3) follows from Lemma~\ref{lem:tr_sym}.

Assume (3). Since we have already proved that (3) implies (1), we can conclude by Lemma~\ref{lem:tr_sym} that there are $g,h\in S^1$
such that $ghg=g$, $hgh=h$, $a''g = gb''$, $gh = a^{\omega}$, and $hg = b^{\omega}$. Thus
\[
a g = a ghg = a a^{\omega} g = a'' g = g b'' = g b^{\omega} b = g hg b = gb\,.
\]
This proves half of (4) and the proof of the other part is similar.
Assume (4). Then $a''g = a^{\omega} ag = a^{\omega} gb = ghgb = g b^{\omega} b = gb''$,
which proves (3).

So far, we have proved (1)$\iff$(2)$\iff$(3)$\iff$(4). In view of Lemma~\ref{lem:tr_sym}, (1) clearly implies (5).

Assume (5). Set $u = gb''$, $v = h$. Then $uv = gb''h = a''$ and
$vu = hgb'' = hgha''g= ha''g=b''$. Thus $a''\cp b''$, which proves (6).

Finally, assume (6). Then $a'' = uv$, $b'' = vu$ for some $u,v\in S^1$, which implies
\begin{equation}
\label{eqn:tr_tmp0}
a'' u = u b''\quad\text{and}\quad b''v=va''\,.
\end{equation}
Since $a' = a''' = (uv)'$ and $b' = b''' = (vu)'$, Lemma~\ref{lem:xyyx} implies
\begin{equation}
\label{eqn:tr_tmp1}
a'u = ub' \quad\text{and}\quad b'v = va'\,.
\end{equation}

Now set $g = a'u$ and $h = b^{\omega} v$. Then
\begin{align*}
 gh &= a'ub^\ome v=a'ub'b''v\by{eqn:tr_tmp0}a'ub'va''\by{eqn:tr_tmp1}a'uva'a''=a' u v a^{\omega} = a' a'' a^{\omega} = a^\ome a^{\omega}=a^\ome\,,\\
 hg &= b^{\omega} v a' u \by{eqn:tr_tmp1} b^{\omega} v u b' = b^{\omega} b'' b' = b^{\omega}\,,\\
 a'' g &= a'' a' u \by{eqn:tr_tmp1} a'' u b' \by{eqn:tr_tmp0} u b'' b'  = u b' b'' \by{eqn:tr_tmp1}a'ub''= g b''\,.
\end{align*}
This proves (3) and completes the proof of the theorem.
\end{proof}

The equivalence of (5) and (6) in Theorem~\ref{thm:trace} was proved for regular semigroups by Kudryavtseva
\cite[Cor.~6 and Thm.~2]{ganna}. The equivalence of (1) and (6) for finite semigroups
can also be extracted from the literature since each is equivalent to the notion of conjugacy defined by having all
characters coincide \cite[Thm.~2.2]{McAl1972}, \cite{Steinberg15}. A direct proof of the equivalence in the finite case
is also straightforward \cite{StEmail}.

If we specialize (1)$\iff$(3)$\iff$(5)$\iff$(6) of Theorem~\ref{thm:trace} to 
completely regular elements, we obtain the following.

\begin{cor}
\label{thm:ganna}
Let $S$ be a semigroup and let $a,b\in \Epi_1(S)$. The following are equivalent:
\begin{enumerate}[label=\textup{(\arabic*)}]
\item $a\ctr b$;
\item $\exists_{g,h\in S^1}\ ag = gb,\ gh = a^{\omega},\ hg = b^{\omega}$;
\item $\exists_{g,h\in S^1}\ ghg=g,\ hgh=h,\ hag = b,\ gbh = a$;
\item $a\cp b$.
\end{enumerate}
\end{cor}
The equivalence of (3) and (4) in Corollary~\ref{thm:ganna} was proved by Kudryavtseva \cite[Thm.~1(1)]{ganna}.

\begin{theorem}
\label{thm:steinbergExtensionEpigroups}
Let $S$ be a semigroup. Then:
\begin{enumerate}[label=\textup{(\arabic*)}]
\item $\ctr$ is an equivalence relation on $\Epi(S)$;
\item for all $x\in \Epi(S)$, $x\ctr x''$;
\item for all $x,y\in S$ such that $xy,yx\in \Epi(S)$, $xy\ctr yx$;
\item $\ctr$ is the smallest equivalence relation on $\Epi(S)$ such that {\rm(2)} and {\rm(3)} hold.
\end{enumerate}
\end{theorem}
\begin{proof}
For (1): The proof of \cite[Prop.~8.2]{Steinberg15} can be repeated verbatim in
this setting.

For (2): Setting $g = x''$, $h = x'$, we have $ghg=g$, $hgh = h$,
$hx''g = x'x''x' = x'' = (x'')''$ and $gh = hg = (x'')^{\omega}=x^{\omega}$.

For (3): Since $(xy)'' = xy(xy)'xy = x\cdot y(xy)'xy$ and
$(yx)'' = yx(yx)'yx = y(xy)'xy\cdot x$ using Lemma~\ref{lem:xyyx}, we have
$(xy)''\cp (yx)''$, and so $xy\ctr yx$ by Theorem~\ref{thm:trace}.

For (4): Suppose $\theta$ is an equivalence relation on $\Epi(S)$ such that $x\ \theta\ x''$ for all $x\in\Epi(S)$
and $xy\ \theta\ yx$ for all $x,y\in S$ such that $xy,yx\in\Epi(S)$. If $a\ctr b$ for some $a,b\in \Epi(S)$, then by
Theorem~\ref{thm:trace}, there exist $u,v\in S^1$ such that $a'' = uv$, $b'' = vu$.
Thus $a\ \theta\ a'' = uv\ \theta\ vu = b''\ \theta\ b$. Therefore $\ctr\,\,\subseteq\,\,\theta$,
as claimed.
\end{proof}

Now we have reached one of our goals of this section, which is to verify the
inclusions on the right side of Figure~\ref{fig0}.

\begin{theorem}
\label{thm:steinbergInclusion}
 Let $S$ be a semigroup. As relations on $\Epi(S)$, the following inclusions hold:
\[
 \cp\ \subseteq\  \cp^*\ \subseteq\ \ctr\ \subseteq\ \co\,.
\]
\end{theorem}
\begin{proof}
The second inclusion follows from Theorem~\ref{thm:steinbergExtensionEpigroups}. The third inclusion
follows from Theorem~\ref{thm:trace}.
\end{proof}

The transitivity of $\cp$ on completely regular elements, a result first obtained by
Kudryavtseva \cite[Cor.~4]{ganna}, now follows easily. We interpret it here as the equality
of certain notions of conjugacy.

\begin{cor}
\label{old4.6}
Let $S$ be a semigroup. As relations on $\Epi_1(S)$, we have $\cp\ =\ \cp^*\ =\ \ctr$.
In particular, $\cp$ is transitive on completely regular semigroups.
\end{cor}
\begin{proof}
This follows from Corollary~\ref{thm:ganna} and Theorem~\ref{thm:steinbergExtensionEpigroups}(1).
\end{proof}

In Corollary~\ref{old4.6}, we cannot include $\co$ among the notions of conjugacy
which coincide. To see this, consider an abelian group with
a zero adjoined. In such a semigroup, $\cp\,\,=\,\,\cp^*\,\,=\,\,\ctr$ is the identity relation,
but $\co$ is the universal relation.

We pointed out in \S\ref{ssym} that for an infinite set $X$, the symmetric inverse semigroup $\mi(X)$ is not an epigroup,
so trace conjugacy is not defined in $\mi(X)$. However, by Theorem~\ref{thm:steinbergExtensionEpigroups},
$\ctr$ is an equivalence relation on $\Epi(\mi(X))$. Using the results of this section,
we can characterize $\ctr$ on $\Epi(\mi(X))$ for
a countably infinite $X$. The following lemma shows that the elements of $\Epi(\mi(X)$ are precisely the transformations
in $\mi(X))$ that don't have any rays and whose lengths of chains are uniformly bounded.
The lemma follows immediately from the fact that $\bt\in\mi(X)$ is an element
of a subgroup of $\mi(X)$ if and only if~$\bt$ is a join of cycles.

\begin{lemma}
\label{lixa}
Let $\al\in\mi(X)$. Then $\al$ is an epigroup element if and only if $\oa=\ua=\la=\emptyset$
and there is a positive integer $n$ such that $\ta^k=\emptyset$ for all $k>n$.
\end{lemma}

Recall that an idempotent $\vep\in\mi(X)$ is completely determined by its domain: for every $x\in\dom(\vep)$,
$x\vep=x$. For $A\subseteq X$, we will denote the idempotent in $\mi(X)$ with domain $A$ by $\vea$.

\begin{lemma}
\label{lixb}
Let $\al\in\Epi(\mi(X))$. Then $\al$ and $\al''$ have the same cycle type.
\end{lemma}
\begin{proof}
By Lemma~\ref{lixa}, $\al$ does not contain any rays and there is a positive integer $n$ such that $\ta^k=\emptyset$ for all $k>n$.
Thus $\al=\join_{\del\in\da}\!\del\jo\join_{\tet\in\ta}\!\tet$ and its cycle-chain type is
\[
\lan|\da^1|,|\da^2|,|\da^3|,\ldots;|\ta^1|,|\ta^2|,\ldots,|\ta^n|\ran.
\]
Then $\al^n$ is in a group $\gh$-class of $\mi(X)$ whose identity is the idempotent
$\al^\ome=\vea$, where $A=\bigcup\{\dom(\del):\del\in\da\}$. Thus
\[
\al''=(\al')^{-1}=((\al\al^\ome)^{-1})^{-1}=\al\vea=\join_{\del\in\da}\!\!\del,
\]
and the result follows.
\end{proof}

\begin{theorem}
\label{ttrix}
Let $X$ be a countably infinite set. Then for all $\al,\bt\in\Epi(\mi(X))$,
$\al\ctr\bt$ if and only if $\al$ and $\bt$ have the same cycle type.
\end{theorem}
\begin{proof}
Let $\al,\bt\in\Epi(\mi(X))$. The following statements are true:
\begin{itemize}
  \item [(a)] $\al\ctr\bt$ if and only if $\al''\cp\bt''$ (by Theorem~\ref{thm:trace});
  \item [(b)] $\al''\cp\bt''$ if and only if $\al''$ and $\bt''$ have the same cycle type (by \cite[Lem.~4]{KuMa07});
  \item [(c)] $\al$ and $\al''$ have the same cycle type, and the same is true for $\bt$ and $\bt''$ (by Lemma~\ref{lixb}).
\end{itemize}
The result clearly follows from  (a)--(c).
\end{proof}

Now we would like to exhibit a larger class of semigroups in which $\cp\ =\ \cp^*\ \subset\ \ctr$,
where the last inclusion is proper. At this point we will no longer work with epigroup elements in
arbitrary semigroups, but rather with epigroups. In particular, this means we will
change our point of view about the role of pseudo-inverses.

Following Petrich and Reilly \cite{PeRe99} for completely regular semigroups and Shevrin \cite{Shevrin} for epigroups,
it is now customary to view an epigroup $(S,\cdot)$ as a \emph{unary} semigroup
$(S,\cdot,{}')$ where $x\mapsto x'$ is the map sending each element to its pseudo-inverse. By a \emph{variety of
epigroups}, we will mean a class of epigroups viewed as a variety of unary semigroups in the usual sense: closed under unary
subsemigroups, homomorphic images and direct products. The class of all epigroups is not a variety because it is not closed
under arbitrary direct products, but the following identities, all of which we have already seen, hold in epigroups:
\begin{align}
x'xx' &= x'\,,      \label{eq:epi1} \\
xx' &= x'x\,,       \label{eq:epi2} \\
xx'x &= x''\,,      \label{eq:epi3} \\
(xy)'x &= x(yx)'\,, \label{eq:epi4} \\
x''' &= x'\,.       \label{eq:epi5}
\end{align}
We note that the class $\mathcal{E}_n$ (that is, the epigroups $S$ such that $S = \Epi_n(S)$) is a variety of epigroups
axiomatized \cite[Prop.~2.10]{Shevrin} by associativity, \eqref{eq:epi1}, \eqref{eq:epi2}, and
\begin{equation}
\label{eq:epi8}
x^{n+1} x' = x^n\,.
\end{equation}

\medskip

Let $\mathcal{W}$ be the class of semigroups $S$ such that the subsemigroup $S^2:=\{ab\mid a,b\in S\}$
is completely regular. This class contains all completely regular semigroups, all null semigroups
(semigroups satisfying the identity $xy=uv$) and, more generally, all variants of completely regular
semigroups. (We will recall the definition of a variant of a semigroup later in the section.)
We first prove that $\mathcal{W}$ is a variety of epigroups.


\begin{prop}
\label{prp:E1_W_E2}
Any semigroup in $\mathcal{W}$ is an epigroup. The following proper inclusions 
of epigroup varieties hold:
$\mathcal{E}_1\subset \mathcal{W}\subset \mathcal{E}_2$.
\end{prop}
\begin{proof}
For $S\in \mathcal{W}$, every $a\in S$ satisfies $a^2\in \Epi_1(S)$, that is, $a^2$ lies in a subgroup
of $S$. Thus $S\in \mathcal{E}_2$, which both verifies the first assertion and the second inclusion.
The second inclusion is also proper, as can be seen by considering
a $3$-element monoid $S = \{e,a,b\}$ where $e$ is the identity element and $\{a,b\}$ is a null
subsemigroup with $xy = a$ for all $x,y\in\{a,b\}$. Then $S$ is clearly in $\mathcal{E}_2$, but
$ea = a$ is not completely regular, so $S$ is not in $\mathcal{W}$.

Finally, the first inclusion is obvious from the definition of $\mathcal{W}$, and since every null semigroup
is in $\mathcal{W}$, so the inclusion is also proper.
\end{proof}

The following result characterizes $\mathcal{W}$ in terms of pseudo-inverses.

\begin{prop}
\label{prp:W}
Let $S$ be a semigroup. Then $S$ is in $\mathcal{W}$ if and only if
$S$ is an epigroup in $\mathcal{E}_2$ satisfying the additional identity
\begin{equation}
\label{eq:xy''}
(xy)'' = xy\,.
\end{equation}
\end{prop}
\begin{proof}
If $S$ is in $\mathcal{W}$, then $S$ is in $\mathcal{E}_2$ by
Proposition~\ref{prp:E1_W_E2}. We have already noted that the completely
regular elements $a$ in an epigroup are characterized by the equation $a'' = a$,
so \eqref{eq:xy''} holds by definition of $\mathcal{W}$.

Conversely, if $S$ is an epigroup in $\mathcal{E}_2$ satisfying \eqref{eq:xy''}, then
combining \eqref{eq:epi3} and \eqref{eq:xy''} shows that each $xy$ lies in a subgroup
of $S$.
\end{proof}

\begin{theorem}
\label{thm:W}
Let $S$ be an epigroup in $\mathcal{W}$. Then $\cp\ =\ \cp^*\ \subset\ \ctr$.
\end{theorem}
\begin{proof}
Suppose $a\cp b$ and $b\cp c$, that is, $a=uv$, $b=vu=xy$, and $c=yx$ for some $u,v,x,y\in S^1$.
If $a=b$ or $b=c$, then clearly $a\cp c$. Otherwise, $a,b,c\in S^2\subseteq \Epi_1(S)$,
so $a\cp c$ by Corollary~\ref{old4.6}. Thus $\cp$ is transitive, and so $\cp\ =\ \cp^*$.

To see that the inclusion $\cp^*\ \subset\ \ctr$ is proper, consider a $2$-element
null semigroup $S = \{a,b\}$ with $xy = a$ for all $x,y\in S$. Then $a' = b' = a$.
As already noted, null semigroups are in $\mathcal{W}$. Since $a'' = b''$, we have
$a\ctr b$ (by Theorem~\ref{thm:trace}), but $a$ and $b$ are evidently not $p$-conjugate.
\end{proof}

To show that the variety $\mathcal{W}$ is of more than just formal interest,
we will now show that it contains all
variants of completely regular semigroups. First, we recall the notion of variant.

Let $S$ be a semigroup and let $a\in S$. Then the pair $(S,\circ)$, where
$\circ$ is a binary operation on $S$ defined by $x\circ y=xay$, is called the
\emph{variant} of $S$ at $a$.
Variants of semigroups are semigroups. Besides giving a construction of new
semigroups from old ones, variants
also provide an interesting interpretation of Nambooripad's natural partial
order on regular semigroups
\cite{Na80}. (See \cite{hickey,hickey2} and also \cite{lawson,ganna2}).

Since $\mathcal{W}$ can be viewed as a variety of unary semigroups, we will
also find it helpful to introduce \emph{unary variants}.
Let $(S,\cdot,{}')$ be a unary semigroup, and fix $a\in S$. Then the unary
semigroup $(S,\circ,^*)$,
where $(S,\circ)$ is the variant of $S$ at $a$ and $x^*=(xa)' x(ax)'$, is
called the \emph{unary variant} of $S$ at $a$.
Since it will always be clear from the context when we mean a unary variant, we
will usually drop the word ``unary'' when referring to variants.

Variants of completely regular semigroups are not, in general, completely regular.
\begin{example}
Let $S=\{0,1\}$ be the $2$-chain. Since $S$ is a semilattice, it is certainly
completely regular. However, its variant at $0$ is the null semigroup, which is not even regular.
\end{example}

\begin{theorem}
\label{thm:variants}
Let $(S,\cdot,{}')$ be a completely regular semigroup, and fix $a\in S$. Let
$(S,\circ,{}^*)$ be the variant of $S$ at $a$, that is,
\[
x\circ y = xay \qquad \text{and}\qquad x^* = (xa)' x (ax)'
\]
for all $x,y\in S$. Then $(S,\circ,{}^*)$ is in $\mathcal{W}$.
\end{theorem}
\begin{proof}
All we need to show is that $S\circ S$ is a subsemigroup of $(S,\circ ,{}^*)$
that is completely regular. We will first prove that $(S,\circ ,{}^*)$ is an
epigroup in  $\mathcal{E}_2$, which implies $S\circ S$ is also an epigroup,
and then show that $S\circ S$ satisfies the identity (\ref{eq:epi8}).

We begin proving that $x^*\circ x\circ x^*=x^*$. Indeed, we have
\begin{align*}
x^*\circ x\circ x^* & =  (xa)'x(ax)' ax\underbrace{a(xa)'}x(ax)'\\
        & \by{eq:epi4}  (xa)'x(ax)'ax\underbrace{(ax)'ax(ax)'} \\
        & \by{eq:epi1}  (xa)'x\underbrace{(ax)'ax(ax)'} \\
        & \by{eq:epi1}  (xa)'x(ax)' = x^*.
\end{align*}
Then we also have $x\circ x^*= x^* \circ x$ since
\begin{align*}
x\circ x^* & = \underbrace{xa(xa)'}x(ax)'\\
        & \by{eq:epi2} (xa)'x\underbrace{ax(ax)'} \\
        & \by{eq:epi2} (xa)'x(ax)'ax = x^*\circ x\,.
\end{align*}
Finally, $x^3\circ x^k=x^2$ since
\begin{align*}
x\circ x\circ x\circ  x^* & = xa\underbrace{xaxa(xa)'}x(ax)'\\
        & \by{eq:epi8} x\underbrace{axax(ax)'}\qquad
\text{($(S,\cdot,{}')$ is completely regular)} \\
        & \by{eq:epi8} xax = x\circ x\qquad
\text{($(S,\cdot,{}')$ is completely regular)},
\end{align*}
and so $(S,\circ,{}^*)$ is an epigroup of $\mathcal{E}_2$.

Given an element $x\circ y$ of $S\circ S$ we will show that $(x\circ y)^2\circ
(x\circ y)^* = x\circ y$. Indeed,
\begin{align*}
 (x^*)^* & =(x\circ y)\circ (x\circ y)\circ (x\circ y)^* \\
        & = \underbrace{xayaxaya(xaya)'}xay(axay)'\\
        & \by{eq:epi8}  xay\underbrace{axay(axay)'} \\
        & \by{eq:epi2}  \underbrace{xa}y(axay)'axay \\
        & \by{eq:epi8}  \underbrace{xa(xa)'}xay(axay)'axay\qquad
\text{($(S,\cdot,{}')$ is completely regular)}, \\
        & \by{eq:epi2}  (xa)'x\underbrace{axay(axay)'axay} \\
        & \by{eq:epi8}  \underbrace{(xa)'xa}xay\qquad
\text{($(S,\cdot,{}')$ is completely regular)}, \\
	& \by{eq:epi2} \underbrace{xa(xa)'xa}y \\
	& \by{eq:epi8} xay=x\circ y   \qquad
\text{($(S,\cdot,{}')$ is completely regular)}. \qedhere
\end{align*}
\end{proof}

\begin{cor}
\label{tcr}
The relation $\cp$ is transitive in every variant of a completely regular semigroup.
\end{cor}

From the preceding result, it is natural to conjecture that if $p$-conjugacy is transitive
in some epigroup, then perhaps the relation is transitive in all of the epigroup's
variants. The following example shows this is not true even for regular epigroups from
$\mathcal{E}_2$.

\begin{example}
\label{4.14}
Let $S$ be the following semigroup, which is both regular and in
$\mathcal{E}_2$:
\[
\begin{array}{r|cccccc}
\cdot & 0 & 1 & 2 & 3 & 4 \\
\hline
    0 & 0 & 0 & 0 & 0 & 0  \\
    1 & 0 & 0 & 0 & 1 & 2  \\
    2 & 0 & 1 & 2 & 1 & 2  \\
    3 & 0 & 0 & 0 & 3 & 4  \\
    4 & 0 & 3 & 4 & 3 & 4
\end{array}
\]
Let $T$ be the variant of $S$ at $1$:
\[
\begin{array}{r|cccccc}
\circ & 0 & 1 & 2 & 3 & 4 \\
\hline
    0 & 0 & 0 & 0 & 0 & 0  \\
    1 & 0 & 0 & 0 & 0 & 0  \\
    2 & 0 & 0 & 0 & 1 & 2  \\
    3 & 0 & 0 & 0 & 0 & 0  \\
    4 & 0 & 0 & 0 & 3 & 4
\end{array}
\]
In $S$, $p$-conjugacy is an equivalence relation that induces the partition
$\{\{0,1\},\{2,3,4\}\}$. However, in $T$, $p$-conjugacy is not transitive
because $2\cp0$ and $0\cp1$, but $(2,1)\notin\,\,\cp$.
\end{example}

Next we will consider epigroups in which all notions of conjugacy on the right side of
Figure~\ref{fig0} coincide. An obvious necessary condition is that $\cp^*\ =\
\cp$, that is,
that $\cp$ must be transitive. Another necessary condition follows from just the assumed
equality of $\ctr$ and $\co$.

\begin{prop}
\label{lem:epi-oc}
Let $S$ be an epigroup in which $\ctr\ =\ \co$. Then $E(S)$ is an antichain.
\end{prop}
\begin{proof}
Suppose $e,f\in E(S)$ satisfy $e \leq f$. Setting $g=h=e$, we have
$eg = ee = e = ef = gf$ and $fh = fe = e = ee = he$. Thus $e \co f$.
Since $\ctr\ =\ \co$, we have $e\ctr f$, and so $e=f$ by Lemma
\ref{lem:epi-ef}. It follows that $E(S)$ is an antichain.
\end{proof}

A natural class of semigroups in which $\cp$ is transitive and idempotents form an antichain
is the class of completely simple semigroups. A semigroup $S$ is \emph{simple} if it has
no proper ideals \cite[p.~66]{Ho95}. A simple semigroup $S$ is called \emph{completely simple}
if it has a primitive idempotent (that is, an idempotent that is minimal with respect to
the partial order $\leq$) \cite[p.~77]{Ho95}. This turns out to be equivalent to
\emph{every} idempotent in $S$ being primitive, that is, the idempotents in $S$ forming an antichain.

A completely simple semigroup can be identified with its Rees matrix 
representation
$\mathcal{M}(G;I,J;P)$, with elements from $I\times G\times \Lambda$, where 
$I$ and $\Lambda$ are nonempty sets, $G$ is a group,
and multiplication is defined by
\begin{equation}
\label{eree}
(i,a,\lam)(j,b,\mu)=(i,ap_{\lam j}b,\mu),
\end{equation}
where $P=(p_{\lam j})$ is a $\Lambda\times I$ matrix with entries in $G$
\cite[Theorem~3.3.1]{Ho95}. From this characterization, it is clear that every element of
a completely simple semigroup is contained in a subgroup, that is, completely simple semigroups
are completely regular.

\begin{theorem}
\label{pcs}
In completely simple semigroups, we have $\cp\ =\ \cp^*\ =\ \ctr\ =\ \co$.
\end{theorem}
\begin{proof}
By Theorem~\ref{thm:steinbergInclusion}, it suffices to prove that $\co\ \subseteq\ \cp$.
To do this, we identify $S$ with its Rees
matrix representation $S=\mathcal{M}(G;I,J;P)$. Let $(i,a,\lam),(j,b,\mu)\in 
S$ and suppose
$(i,a,\lam)\co(j,b,\mu)$. Then, by \eqref{eree}, there exist $(i,c,\mu),(j,d,\lam)\in S$ such that
\[
(i,a,\lam)(i,c,\mu)=(i,c,\mu)(j,b,\mu)\,\,\mbox{ and
}\,\,(j,b,\mu)(j,d,\lam)=(j,d,\lam)(i,a,\lam),
\]
which implies
\begin{equation}
\label{jana}
ap_{\lam i}c=cp_{\mu j}b\,\,\mbox{ and }\,\,bp_{\mu j}d=dp_{\lam i}a.
\end{equation}
Consider $x = (d p_{\lam i})\inv b$ and $y = d$. Then, by \eqref{jana},
\begin{alignat*}{2}
(i,x,\mu)(j,y,\lam)&=(i,xp_{\mu j}y,\lam)=(i,(d p_{\lam i})\inv bp_{\mu
j}d,\lam)=(i,(d p_{\lam i})\inv dp_{\lam i}a,\lam)=(i,a,\lam),\\
(j,y,\lam)(i,x,\mu)&=(j,yp_{\lam i}x,\mu)=(j,dp_{\lam i}(d p_{\lam i})\inv
b,\mu)=(j,b,\mu),
\end{alignat*}
which implies $(i,a,\lam)\cp (j,b,\mu)$.
\end{proof}

\begin{theorem}
\label{main7}
Let $S$ be a regular epigroup without zero. The following are equivalent:
\begin{enumerate}[label=\textup{(\arabic*)}]
\item $\cp\,\,=\,\,\co$ in $S$;
\item $S$ is completely simple.
\end{enumerate}
\end{theorem}
\begin{proof}
Suppose $\co\,\,=\,\,\cp$ in $S$.
Since $S$ is an epigroup, we also have $\ctr\,=\,\co$, and thus $E(S)$ is an antichain
by Proposition \ref{lem:epi-oc}, that is, every idempotent in $S$ is primitive.
Since $S$ is also regular, we conclude that $S$ is completely simple \cite[Thm.~3.3.3]{Ho95}.

The converse follows from Theorem~\ref{pcs}.
\end{proof}

\begin{theorem}
\label{main6}
Let $S$ be an epigroup in $\mathcal{W}$ without zero. The following are
equivalent:
\begin{enumerate}[label=\textup{(\arabic*)}]
\item $\cp\,\,=\,\,\co$ in $S$;
\item $S$ is completely simple.
\end{enumerate}
\end{theorem}
\begin{proof}
Suppose $\co\,\,=\,\,\cp$ in $S$. Arguing as in the preceding proof, we have
that every idempotent in $S$ is primitive.
Next,  since $\co\,\,=\,\,\cp$,  we have $x\cp x''$, by
Theorem~\ref{thm:steinbergExtensionEpigroups} and
Theorem~\ref{thm:steinbergInclusion}. Hence there exist $u,v\in S^1$
such that $x'' = uv$ and $x = vu$. But then $x'' = (vu)'' = vu = x$, using
\eqref{eq:xy''} (since $S$ is in $\mathcal{W}$).
Therefore $x$ is completely regular. It follows that $S$ is completely regular.
Finally, since $S$ is completely regular and every idempotent is primitive in
$S$, it follows that $S$
is completely simple \cite[Thm.~3.3.3]{Ho95}.

The converse follows from Theorem~\ref{pcs}.
\end{proof}

We now give two examples of \emph{inverse epigroups} (epigroups that are
also inverse semigroups) to illustrate some possible relations between the conjugacies in the variety $\mathcal{E}_2$.

\begin{example}
\label{ex:E2_a}
In a semigroup from the epigroup variety $\mathcal{E}_2$, we can have
$\cp\ \subset\ \cp^*\ =\ \ctr\ \subset\ \con\ =\ \co$, where the inclusions are
strict. (In particular, $\cp$ need not be transitive in a semigroup from $\mathcal{E}_2$.)
Consider, for example, the inverse semigroup $S$ given by the following
multiplication table.
\[
\begin{array}{r|cccccc}
\cdot & 0 & 1 & 2 & 3 & 4 & 5\\
\hline
    0 & 0 & 0 & 0 & 3 & 3 & 3 \\
    1 & 0 & 1 & 0 & 3 & 4 & 3 \\
    2 & 0 & 0 & 2 & 3 & 3 & 5 \\
    3 & 3 & 3 & 3 & 0 & 0 & 0 \\
    4 & 3 & 3 & 4 & 0 & 0 & 1 \\
    5 & 3 & 5 & 3 & 0 & 2 & 0
\end{array}
\]
This is an $E$-unitary inverse semigroup.
(An inverse semigroup $S$ is \emph{$E$-unitary} if for all $e,a\in S$, if $e$
and $ea$ are idempotents, then $a$ is an idempotent.)
This semigroup is in $\mathcal{E}_2$ since every entry on the main diagonal of the
table is an idempotent, but it is not Clifford (that is both completely regular and inverse),
not even in $\mathcal{W}$, which
can be checked directly, but also
follows because $p$-conjugacy in $S$ is not transitive. Indeed,
we have $4\cp 3$ (since $4 = 1\cdot 4$ and $3=4\cdot 1$) and $3\cp 5$ (since
$3=1\cdot5$ and $5=5\cdot 1$), but there are no
$x,y$ such that $4=xy$ and $5=yx$. It is straightforward to check that
$\cp$ is the symmetric and reflexive closure of $\{(1,2),(3,4),(3,5)\}$, that
$\cp^*\ =\ \ctr$, and that $\con\ =\ \co$ has equivalence classes $\{0,1,2\}$ and
$\{3,4,5\}$. Thus we have the claimed strict inclusions.
\end{example}

\begin{example}
\label{ex:E2_b}
There are epigroups in $\mathcal{E}_2$ but not $\mathcal{W}$ in which
$p$-conjugacy is transitive.
Consider, for example, the following inverse semigroup $S$,
which is an ideal extension of the group $\{1,a\}$ by the Brandt semigroup
$\{0,b,c,e,f\}$  \cite[p.~152]{Ho95}:
\[
\begin{array}{c|ccccccc}
\cdot & 1 & a & 0 & b & c & e & f\\
\hline
    1 & 1 & a & 0 & b & c & e & f \\
    a & a & 1 & 0 & e & f & b & c \\
    0 & 0 & 0 & 0 & 0 & 0 & 0 & 0 \\
    b & b & f & 0 & 0 & f & b & 0 \\
    c & c & e & 0 & e & 0 & 0 & c \\
    e & e & c & 0 & 0 & c & e & 0 \\
    f & f & b & 0 & b & 0 & 0 & f
\end{array}
\]
The semigroup $S$ is an $E^*$-unitary inverse monoid.
(An inverse semigroup $S$ with zero is \emph{$E^*$-unitary} if for all $e,a\in
S$, if $e$ and $ea$ are nonzero idempotents, then $a$ is an idempotent.)
Again, $S$ is in $\mathcal{E}_2$ since
every entry on the main diagonal of the table is an idempotent, but it
is not Clifford because neither $b$ nor $c$ are completely regular, not even in
$\mathcal{W}$ because, for instance, $a\cdot e = b$.

However, this time, $\cp$ is an equivalence relation, with the equivalence classes
$\{1\}$, $\{a\}$, $\{0,b,c\}$, and $\{e,f\}$. Also $\cp\ =\ \ctr$.
This semigroup, incidentally,
is the smallest example of an inverse semigroup that is not completely regular
but in which $p$-conjugacy is transitive.
Note that $\con$ has equivalence classes  $\{1\}$, $\{a\}$, $\{0\}$, $\{b,c\}$,
and $\{e,f\}$, and therefore $\con\ \subset \ \cp$.
\end{example}

Let us now turn our attention to semigroups with zero.
A semigroup $S$ with zero is $0$-\emph{simple} if $S^2\ne\{0\}$ and $\{0\}$
and $S$ are the only ideals of $S$ \cite[p.~66]{Ho95}.
A $0$-simple semigroup $S$ is called \emph{completely $0$-simple} if it
contains a primitive idempotent \cite[p.~70]{Ho95}.
A completely $0$-simple semigroup $S$ can be identified with its Rees matrix
representation
$\mathcal{M}^0(G;I,\Lambda;P)$, with elements from $(I\times 
G\times\Lambda)\cup\{0\}$, where $I$ and $\Lambda$ are nonempty
sets, $G$ is a group, and multiplication is defined by
$(i,a,\lam)(j,b,\mu)=(i,ap_{\lam j}b,\mu)$ if $p_{\lam j}\ne 0$,
$(i,a,\lam)(j,b,\mu)=0$ if $p_{\lam j}=0$,
and $(i,a,\lam)0=0(i,a,\lam)=0$,
where $P=(p_{\lam j})$ is a $\Lambda\times I$ matrix with entries in
$G\cup\{0\}$ such that no row or column of $P$ consists entirely of zeros
\cite[Theorem~3.2.3]{Ho95}.

Theorem~\ref{main6} does not remain true if $\co$ is replaced with $\con$ and
``completely simple''
with ``completely $0$-simple.'' Indeed, suppose that in the matrix $P$, we have
$p_{\lam j}\ne0$ and $p_{\mu i}=0$. Let $a,b\in G$.
Then $(i,a,\lam)(j,b,\mu)=(i,ap_{\lam j}b,\mu)\ne0$ and $(j,b,\mu)(i,a,\lam)=0$.
Thus $(i,ap_{\lam j}b,\mu)\cp0$, while $(i,ap_{\lam j}b,\mu)$ and $0$ are not
$\con$-related
since in every semigroup with zero, the $c$-conjugacy  class of $0$ is $\{0\}$
\cite[Lemma~2.3]{AKM14}. Hence
$\con\,\,\ne\,\,\cp$ in completely $0$-simple semigroups.

We have, however, the following results.

\begin{prop}
\label{pcs0}
 For a completely $0$-simple semigroup $\mathcal{M}^0(G;I,\Lambda;P)$,
we have $\con\ \subseteq\ \cp$.
Moreover, $\con\ =\ \cp$ if and only if the sandwich matrix $P$ has only
nonzero elements.
\end{prop}
\begin{proof}
Let $(i,a,\lambda),(j,b,\mu)$ be non zero elements of 
$S=\mathcal{M}^0(G;I,\Lambda;P)$ such that
$(i,a,\lambda)\con (j,b,\mu)$. By (\ref{e1dcon}) there exist  nonzero elements
$(i,c,\mu),(j,d,\lam)$ with $p_{\lam i},\neq 0$, $p_{\mu j}\neq 0$ such that
\[
(i,a,\lam)(i,c,\mu)=(i,c,\mu)(j,b,\mu)\,\,\mbox{ and
}\,\,(j,b,\mu)(j,d,\lam)=(j,d,\lam)(i,a,\lam).
\] Using the
same arguments as in the proof of Theorem~\ref{pcs}, we obtain
$(i,a,\lam)\cp(j,b,\mu)$.

Now, suppose first that $\con\ = \ \cp$. By the argument showing that $\con\ \ne\ \cp$ in completely $0$-simple
semigroups (see the paragraph above this proposition),
we can conclude that whenever
$p_{\mu i}=0$, for some $i\in I$ and $\mu\in \Lambda$, then $p_{\lambda j}=0$,
for all $j\in I$ and $\lambda \in \Lambda$.

Conversely, suppose that the sandwich matrix $P$ has only nonzero elements.
Then $S$ is isomorphic to $T^0$ where $T$ is the completely simple semigroup
$\mathcal{M}(G;I,\Lambda;P)$. But then, by Theorem~\ref{pcs}, $\cp^T\ = \ \co^T$ in
$T$. Since $S$ has no zero divisors we have in  $\con^S\ = \{(0,0)\}\ \cup \co^T$
and $\{0\}$ is one of the $p$-conjugacy classes of $S$. Therefore,
$\con^S\ = \{(0,0)\}\ \cup \cp^T\ = \ \cp^S$.
\end{proof}

\begin{lemma}
\label{lem:c_anti}
Let $S$ be an epigroup with zero and suppose $\con\ \subseteq\ \ctr$.
Then $E(S)\backslash \{0\}$ is an antichain.
\end{lemma}
\begin{proof}
Suppose $e, f \in E(S)$ with $0 \neq e\leq f$. Since $e$ is in both
$\mathbb{P}^1(e)$ and $\mathbb{P}^1(f)$,
$ee=e=ef$ and $fe=e=ee$, we have $e \con f$.
Since $\con\ \subseteq\ \ctr$, we have $e\ctr f$, and so $e=f$ by
Lemma \ref{lem:epi-ef}. It follows that $E(S)\backslash \{0\}$ is an antichain.
\end{proof}

A semigroup $S$ with zero is called a \emph{$0$-direct union} of completely
$0$-simple semigroups if $S=\bigcup_{i\in I} S_i$,
where each $S_i$ is a completely $0$-simple semigroup and $S_i\cap
S_j=S_iS_j=\{0\}$ if $i\ne j$ \cite[pp.~79--80]{Ho95}.

\begin{theorem}
\label{thm:c0s}
Let $S$ be a regular epigroup with zero.
The following are equivalent:
\begin{enumerate}[label=\textup{(\arabic*)}]
\item $\con\ \subseteq\ \cp$;
\item $\con\ \subseteq\ \ctr$;
\item $S$ is a $0$-direct union of completely $0$-simple semigroups.
\end{enumerate}
\end{theorem}
\begin{proof}
(1)$\implies$(2) is true because $\cp\ \subseteq\ \ctr$ in any epigroup.

Assume (2). By Lemma \ref{lem:c_anti}, every nonzero idempotent is primitive.
Since $S$ is also regular, then by \cite[Thm.~3.3.4]{Ho95}, we obtain (3).

Now assume (3), that is,
$S=\bigcup_{i\in I} S_i$, where each $S_i$ is a completely $0$-simple semigroup
and $S_i\cap S_j=S_iS_j=\{0\}$ if $i\ne j$.

We will show that if $a\con b$ in $S$, then both $a$ and $b$ belong to
the same subsemigroup $S_i$, for some $i\in I$, and that $a\con^{S_i} b$ in
$S_i$. Since the $c$-conjugacy class of $0$ is $\{0\}$, we may assume that $a,b\ne0$.
Suppose $a\con b$ in $S$. By (\ref{e1dcon}) there exist nonzero
elements $g,h\in S$, with $ag\neq 0$ and $bh\neq 0$, satisfying
$ ag=gb$ and $bh=ha$. Thus, since $S_iS_j=\{0\}$, for all $i,j\in I$, we
conclude that $a,b,g,h\in S_i$, for some $i\in I$. But then $a\con^{S_i} b$ in
$S_i$.

So any two $c$-conjugate elements in $S$ are $c$-conjugate elements in a
completely $0$-simple semigroups $S_i$. Hence, by Proposition~\ref{pcs0}, any
two $c$-conjugate elements are also $p$-conjugate in $S_i$, for some $i\in I$.
Since $p$-conjugate elements in $S_i$
are also $p$-conjugate in $S$, we have $\con\ \subseteq\ \cp$
in $S$.
\end{proof}

\medskip

In the last part of this section, we will examine $o$-conjugacy in all
epigroups and $c$-conjugacy in the variety $\mathcal{W}$.

If two elements $a,b$ with $a\neq b$ of a semigroup are $o$-conjugate, say, $ag = gb$ and
$bh = ha$, then in general, there is no apparent connection between $g$ and $h$ beyond these two equations.
In a group, of course, one may assume without loss of generality that $h = g\inv$. The next result shows that
in epigroups, we may similarly restrict the choice of conjugators for $\co$ without loss of generality.

\begin{theorem}
\label{thm:strong-o}
Let $S$ be an epigroup and suppose $a\sim_o b$ for some $a,b\in S$. Then there exist mutually inverse
$g,h\in S^1$ such that $ag = gb$ and $bh = ha$.
\end{theorem}
\begin{proof}
Since $a\co b$, there exist $c,d\in S^1$ such that $ac = cb$ and $bd = da$.
These imply $acd = cda$, a fact we will use without comment in what follows. Set
\begin{equation}
\label{eq:gh}
h = da(cda)' \qquad\text{and}\qquad g = chc\,.
\end{equation}
Then $hch = da(cda)' cda (cda)' \by{eq:epi1} da(cda)' = h$. Thus $h$ is regular and so an inverse of $h$ is given by
$chc = g$, that is, $g$ and $h$ are mutually inverse as claimed.

Now we check that $g$ and $h$ are conjugators of $a$ and $b$. First, we have
\[
bh = \underbrace{bd}a(cda)' = da\underbrace{a(cda)'} \by{eq:epi4} da(acd)'a  = 
ha\,.
\]
Then we use this in the third step of the following calculation:
\[
ag = \underbrace{ac}hc = c\underbrace{bh}c = ch\underbrace{ac} = chcb = gb\,.
\]
This completes the proof
\end{proof}

\begin{example}
In the completely regular case, it is not possible, in general, to choose the
mutually inverse $g$ and $h$ of Theorem~\ref{thm:strong-o} to be $g$ and $g' = g\inv$,
the commuting inverse of $g$. To see this, consider a
$2$-element left zero semigroup $\{a,b\}$. Since $aba = a$, $bab=b$, $a$ and
$b$ are mutually inverse. We
also have $aa = ab$ and $bb = ba$, so $a\co b$. However, $a'=a$ and $b'=b$, so
we cannot have both $ax=xb$ and $bx'=x'a$ for either $x=a$ or $x=b$.
\end{example}

Now we consider $c$-conjugacy. We do not know if there is a full analog of
Theorem~\ref{thm:strong-o} for all epigroups, but there is one for our variety $\mathcal{W}$.
First we need the following result.

\begin{lemma}
\label{lem:W_middle}
Let $S$ be an epigroup with zero in $\mathcal{W}$. If $st = 0$ for some $s,t\in
S^1$, then $sxt = 0$ for all $x\in S^1$.
\end{lemma}
\begin{proof}
First,
\[
ts \by{eq:xy''} (ts)'' \by{eq:epi3} ts\underbrace{(ts)'ts} \by{eq:epi2}
t\underbrace{st}s(ts)' = 0\,.
\]
Then
\[
sxt \by{eq:xy''} (sxt)'' \by{eq:epi3} sxt\underbrace{(sxt)'sxt} \by{eq:epi2}
sx\underbrace{ts}xt(sxt)' = 0\,. \qedhere
\]
\end{proof}

\begin{theorem}
\label{thm:strong-c}
Let $S$ be an epigroup with zero in $\mathcal{W}$ and suppose $a\con b$ for
some
$a,b\in S$. Then there exist mutually inverse
$g\in \mathbb{P}^1(a)$, $h\in \mathbb{P}^1(b)$ such that $ag = gb$ and $bh =
ha$.
\end{theorem}
\begin{proof}
We may assume $a,b\neq 0$.
Since $a\con b$, there exist $c\in \mathbb{P}^1(a)$, $d\in \mathbb{P}^1(b)$
such that $ac = cb$ and $bd = da$.
As before, we will use $acd = cda$ without comment.

Define $h$ and $g$ by \eqref{eq:gh}. By the proof of
Theorem~\ref{thm:strong-o},
we have that $g,h$ are mutually
inverse and satisfy $ag=gb$, $bh=ha$. What remains is to show $h\in
\mathbb{P}^1(b)$ and $g\in \mathbb{P}^1(a)$.

Suppose $(mb)h = 0$ for some $m\in S^1$. We wish to prove $mb=0$. By
Lemma~\ref{lem:W_middle}, $mxbh=0$ for all $x\in S^1$,
and so in particular, we have $mcbh = 0$. Thus $0 = mc\underbrace{bh} = mcha =
mcda(cda)'a$. Multiply both sides on the
right by $cd$ to get
\[
0 = mcda(cda)'\underbrace{acd} = mcda(cda)'cda \by{eq:epi3} m(cda)''
\by{eq:xy''} mc\underbrace{da} = mcbd\,.
\]
Now since $d\in \mathbb{P}^1(b)$, the result of this last calculation implies
$mcb = 0$. Thus $0 = mcb = mac$.
Since $c\in \mathbb{P}^1(a)$, we conclude that $ma = 0$. Using
Lemma~\ref{lem:W_middle} once again, $mxa = 0$ for all $x\in S^1$.
In particular, we have $0 = mda = mbd$. Since $d\in \mathbb{P}^1(b)$, we obtain
$mb=0$ as claimed.

Finally suppose $(ma)g = 0$ for some $m\in S^1$. We wish to prove $ma = 0$. Thus
\[
0 = mag = m\underbrace{ac}hc = mc\underbrace{bh}c = mchac\,.
\]
Since $c\in \mathbb{P}^1(a)$, we have $mcha = 0$, that is, $mcbh = 0$. Since
$h\in \mathbb{P}^1(b)$,
$0 = mcb = mac$. Using $c\in \mathbb{P}^1(a)$ one last time, we get $ma = 0$ as
claimed.
\end{proof}

\begin{example}
The proof of Theorem~\ref{thm:strong-c} depends heavily on the epigroup $S$
being in the variety $\mathcal{W}$, and
indeed the method of proof does not work for all epigroups in general. For
example, consider the commutative monoid
$S$ with zero defined by the following multiplication table:
\[
\begin{array}{r|cccccc}
\cdot & 0 & 1 & 2 & 3 & 4 & 5\\
\hline
    0 & 0 & 0 & 0 & 0 & 0 & 0 \\
    1 & 0 & 1 & 2 & 3 & 4 & 5 \\
    2 & 0 & 2 & 0 & 0 & 2 & 2 \\
    3 & 0 & 3 & 0 & 0 & 2 & 2 \\
    4 & 0 & 4 & 2 & 2 & 5 & 5 \\
    5 & 0 & 5 & 2 & 2 & 5 & 5
\end{array}
\]
This is an epigroup with pseudo-inverse given by $0'=2'=3' = 0$, $1'=1$,
$4'=5'=5$. It is easy to see that $S$ is
in $\mathcal{E}_2$ since every element on the diagonal is an idempotent. $S$ is
not in $\mathcal{W}$ because,
for instance, $(2\cdot 4)'' = 2'' = 0\neq 2\cdot 4$. If $a = 2$, $b=3$,
$c=d=4$,
then $ac=cb$, $bd=da$,
$c\in \mathbb{P}^1(a)$ and $d\in \mathbb{P}^1(b)$. Thus $a\con 3$. Note that
$c$
is not regular, but if
we try to define $g,h$ by \eqref{eq:gh}, we get $g = h = 0$. Thus the proof of
Theorem~\ref{thm:strong-c} does
not apply here. However, note that by setting $g=h=5$, we do obtain mutually
inverse $g,h$ which will suffice.
Therefore in this example, the conclusion of Theorem~\ref{thm:strong-c} is
still
correct.
\end{example}

\section{Comparison results}
\label{scom}
\setcounter{table}{0}
\setcounter{equation}{0}

In this section, we compare the four notions of conjugacy under
discussion in various settings.
In every semigroup,
$\cp\,\,\subseteq\,\,\cp^*\,\,\subseteq \,\, \ctr \,\,\subseteq\,\,\co$ and $\con\,\,\subseteq\,\,\co$.

Regarding $\cp$ and $\con$, we have the following result.

\begin{theorem}
For each of the following conditions:
\begin{enumerate}[label=\textup{(\alph*)}]
 \item $\con\,\,\subset\,\,\cp$,
 \item $\cp\,\,\subset\,\,\con$,
 \item $\cp$ and $\con$ are not comparable with respect to
inclusion,
\end{enumerate}
there exists a semigroup with zero in which the condition holds.
\end{theorem}
\begin{proof}
Proposition~\ref{pcpf} shows that $\con\,\,\subset\,\,\cp$ in any symmetric
inverse semigroup $\mi(X)$ where $2\leq |X| < \infty$.

Example~\ref{ex:E2_a} provides an example of a semigroup $S$ without zero in
which $\cp\ \subset\ \cp^*\ \subset\ \con\,\,=\,\,\co$. Let $S^0$ denote the
semigroup obtained from $S$ by
adding an extra element $0$ acting as a zero.
Then $\cp^{S^0}\,\,=\,\, \cp^S\cup\,\,\{(0,0)\}$ and
$\con^{S^0}\,\,=\,\,\con^S\cup\,\,\{(0,0)\}$. Thus,
$S^0$ is a semigroup with zero in which $\cp\,\,\subset\,\,\con$ as claimed.

Finally, by Proposition~\ref{pcpi}, relations $\cp$ and $\con$ are not
comparable with respect to
inclusion in the symmetric inverse semigroup $\mi(X)$ on a countably infinite set
$X$. There are also
finite semigroups in which $\cp$ and $\con$ are not comparable. Indeed, let $S
=
\{0,1,2,3,4\}$ be the
monoid given by the following multiplication table:
\[
\begin{array}{r|cccccc}
\cdot & 0 & 1 & 2 & 3 & 4 \\
\hline
    0 & 0 & 0 & 0 & 0 & 0  \\
    1 & 0 & 1 & 2 & 3 & 4  \\
    2 & 0 & 2 & 0 & 2 & 0  \\
    3 & 0 & 3 & 4 & 3 & 4  \\
    4 & 0 & 4 & 0 & 4 & 0
\end{array}
\]
It is straightforward to check that $1\con 3$ and all other $\con$-classes are
trivial, while $2\cp 4$ and all other $\cp$-classes are trivial.
\end{proof}

Regarding $\ctr$ and $\con$, we have the following result.
\begin{theorem}
For each of the following conditions:
\begin{enumerate}[label=\textup{(\alph*)}]
 \item $\con\,\,\subset\,\,\ctr$,
 \item $\ctr\,\,\subset\,\,\con$,
 \item $\ctr$ and $\con$ are not comparable with respect to
inclusion,
\end{enumerate}
there exists a semigroup with zero in which the condition holds.
\end{theorem}
\begin{proof}
The following semigroup (\texttt{SmallSemigroup(4,22)} of \cite{Smallsemi}) satisfies condition (a):
\[
\begin{array}{r|cccccc}
\cdot & 0 & 1 & 2 & 3 \\
\hline
    0 & 0 & 0 & 0 & 0   \\
    1 & 0 & 0 & 0 & 1   \\
    2 & 0 & 0 & 0 & 1  \\
    3 & 0 & 1 & 1 & 3
\end{array}
\] $\ctr\ =\{\{0,1,2\},\{3\}\}$ and $\con\ =\{\{0\},\{1,2\},\{3\}\}$.

The following semigroup (\texttt{SmallSemigroup(4,113)} of \cite{Smallsemi}) satisfies condition (b):
\[
\begin{array}{r|cccccc}
\cdot & 0 & 1 & 2 & 3 \\
\hline
    0 & 0 & 0 & 0 & 0   \\
    1 & 0 & 1 & 1 & 1   \\
    2 & 0 & 1 & 2 & 1  \\
    3 & 0 & 3 & 3 & 3
\end{array}
\] $\ctr\ =\{\{0\},\{2\},\{1,3\}\}$ and $\con\ =\{\{0\},\{1,2,3\}\}$.

Finally, the following semigroup (\texttt{SmallSemigroup(4,56)} of \cite{Smallsemi}) satisfies condition (c):
\[
\begin{array}{r|cccccc}
\cdot & 0 & 1 & 2 & 3 \\
\hline
    0 & 0 & 0 & 0 & 0   \\
    1 & 0 & 0 & 0 & 0   \\
    2 & 0 & 0 & 2 & 2  \\
    3 & 0 & 0 & 2 & 3
\end{array}
\] $\ctr\ =\{\{0,1\},\{2\},\{3\}\}$ and $\con\ =\{\{0\},\{1\},\{2,3\}\}$.
\end{proof}

Our next result separates $c$-conjugacy and $o$-conjugacy.
As already mentioned  $\co$ is the universal relation in any semigroup with
zero and  $\con\,\,=\,\,\co$ in any semigroup without zero. Therefore, a trivial
way of separating $\con$ and $\co$ is to consider any semigroup without zero and
then adjoin a zero to that semigroup.

Less trivially, we can separate $\con$ and $\co$ in semigroups with proper zero
divisors. The next theorem shows that the two notions might be different in such a
semigroup in the most extreme way.

\begin{theorem}
\label{tach}
In a semilattice $S$ that is an anti-chain with $0$ and $1$, $\co$ is
universal, while $\con$ is the identity.
\end{theorem}
\begin{proof}
\begin{figure}[ht]
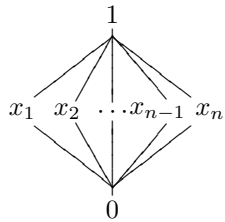

\[
\xy
(20,-3)*{0};
(20,23)*{1};
(20,0)*{}="0";
(20,20)*{}="1";
(11,13)*{}="at";
(29,13)*{}="zt";
(16,13)*{}="bt";
(26,13)*{}="wt";
(20,13)*{}="ct";
(11,7)*{}="ab";
(29,7)*{}="zb";
(16,7)*{}="bb";
(26,7)*{}="wb";
(20,7)*{}="cb";
(8,10)*{x_1}="a";
(33,10)*{x_n}="z";
(14,10)*{x_2}="b";
(26,10)*{x_{n-1}}="w";
(20,10)*{\cdots}="c";
"0";"ab" **\crv{} ?>* \dir{-};
"0";"bb" **\crv{} ?>* \dir{-};
"0";"wb" **\crv{} ?>* \dir{-};
"0";"zb" **\crv{} ?>* \dir{-};
"1";"at" **\crv{} ?>* \dir{-};
"1";"bt" **\crv{} ?>* \dir{-};
"1";"wt" **\crv{} ?>* \dir{-};
"1";"zt" **\crv{} ?>* \dir{-};
"1";"ct" **\crv{} ?>* \dir{-};
"0";"cb" **\crv{} ?>* \dir{-};
\endxy
\]
\caption{Bounded anti-chain.}\label{f2}
\end{figure}

Observe that $\pp^1(1)=\{1\}$, $\pp^1(0)=\{0\}$, and $\pp^1(x)=\{x,1\}$ for all
$x\in S\setminus \{0,1\}$. Therefore, in this semigroup $\con$
is the identity, while $\co$ is the universal relation.
\end{proof}

The same result holds for every null semigroup.
Table~\ref{tb1} was produced using the \emph{Smallsemi} package for GAP
\cite{Smallsemi}. It contains data illustrating how common the extreme behavior of $\con$
is in monoids with zero divisors. In Table~\ref{tb1}, $|S|$ is the order of the semigroup; the column
labeled by ``\# of monoids with $0$-divisors'' contains the number of monoids of order $|S|$ that
have a zero and zero divisors;
the column ``$\con$ is the identity'' contains the number of such monoids in
which $\con$ is the identity relation;
the column ``$\con$ is `universal' '' contains the number of such monoids in
which all nonzero elements form a single conjugacy class.

\begin{table}
\begin{center}
\begin{tabular}{ccccc}
\hline
$|S|$& \# of monoids with $0$-divisors&$\con$ is the identity&$\con$ is
`universal' \\ \hline \hline
      3&1 &1 &0\\ \hline
     4&7&3&1 \\ \hline
     5&58&14&7 \\ \hline
     6& 574 &115&74\\ \hline
  7&8742&3016&972\\ \hline
\end{tabular}
\caption{$c$-conjugacy in small monoids with zero divisors}\label{tb1}
\end{center}
\end{table}

For a large proportion of  the monoids from Table~\ref{tb1}, $c$-conjugacy is the identity.
Observe that in groups, conjugacy is the identity relation if and only if the group is
abelian. This is not the case for $c$-conjugacy in monoids, as the following monoid
with proper divisors of zero shows:
\[
\begin{tabular}{r|rrrrr}
$\cdot$ & 0 & 1 & 2 & 3 & 4\\
\hline
    0 & 0 & 0 & 0 & 0 & 0 \\
    1 & 0 & 0 & 0 & 0 & 1 \\
    2 & 0 & 0 & 0 & 0 & 2 \\
    3 & 0 & 0 & 1 & 0 & 3 \\
    4 & 0 & 1 & 2 & 3 & 4
\end{tabular}
\]
Every element in this monoid is only $c$-conjugate to itself, and the monoid is
not commutative.
This monoid is \texttt{SmallSemigroup(5,110)} in the \emph{Smallsemi} package
for GAP
\cite{Smallsemi}

\smallskip

However, the result analogous to group conjugacy holds for $p$-conjugacy.

\begin{theorem}
\label{tpid}
Let $S$ be a semigroup. Then, $\cp$ is the identity relation in $S$ if and only
if $S$ is commutative.
\end{theorem}
\begin{proof}
If $S$ is commutative and if $x=uv$ and $y=vu$, then obviously $x=uv=vu=y$,
and so $\cp$ is the identity relation.
Conversely, suppose each element of $S$ is $p$-conjugate only to itself.
For all $a,b\in S$, $ab\cp ba$, and so $ab=ba$ by the assumption.
\end{proof}

\begin{theorem}
\label{ttrid}
Let $S$ be an epigroup. Then, $\ctr$ is the identity relation in $S$ if and only
if $S$ is a commutative completely regular epigroup.
\end{theorem}
\begin{proof}
Suppose first that $\ctr$ is the identity relation. Since $\cp\,\subseteq\,\ctr$, it follows that
$\cp$ is also the identity relation, and hence, by Theorem~\ref{tpid}, $S$ is commutative.
In every epigroup, we have $a\ctr a''$ by Theorem~\ref{thm:steinbergExtensionEpigroups}. Since
$\ctr$ coincides with equality, we have $a=a''$ for all $a\in S$. Thus $S$ is a commutative completely
regular epigroup (or, equivalently, a commutative inverse epigroup).

Conversely, if $S$ is a commutative completely regular epigroup, then $\ctr\,=\,\cp$ by
Corollary~\ref{thm:ganna}, and so $\ctr$ is the identity relation by Theorem~\ref{tpid}.
\end{proof}

The corresponding result for $o$-conjugacy is as follows.

\begin{theorem}
\label{thm:o-comm}
Let $S$ be a semigroup. Then:
\begin{enumerate}[label=\textup{(\arabic*)}]
 \item if $S$ is commutative, then $\co$ is the minimum
cancellative congruence on $S$;
 \item $\co$ is the identity relation in $S$ if and only if $S$
is commutative and cancellative.
\end{enumerate}
\end{theorem}
\begin{proof}
For (1): Suppose $S$ is commutative. Then for all $a,b\in S$, $a\co b$ if
and only if $ag=bg$ for some $g\in S^1$. Thus, whenever $a\co b$ we have $ca\co
cb$ and $ac\co bc$, for all $c\in S$, which implies that $\co$ is a congruence.
Denote the congruence class of $x\in S$ by $\bar{x}$. Let
$\bar{a},\bar{b},\bar{c}\in S/\!\!\co$ and suppose
$\bar{a}\bar{b}=\bar{a}\bar{c}$. Then $ab\co ac$, and so $(ab)g=(ac)g$ for some
$g\in S^1$. Since
$S$ is commutative, we have $b(ag)=c(ag)$, and so $b\co c$. Hence
$\bar{b}=\bar{c}$, which implies that
$S/\!\!\co$ cancellative.
Now let $\theta$ be any
cancellative congruence on $S$ and suppose $a\co b$, where $a,b\in S$. Then $ag
= bg$ for some $g\in S^1$, so
$ga\,\,\theta\,\,gb$. Since $\theta$ is cancellative, it follows that
$a\,\,\theta\,\,b$. Therefore,
$\co\,\,\subseteq\theta$, which proves $\co$ is the minimum cancellative
congruence on $S$.

For (2): If $S$ is commutative and cancellative, then (1) implies $\co$ must be
the identity relation. For the converse, note that $xy\co yx$ in any semigroup (since
$(xy)x = x(yx)$  and  $(yx)y = y(xy)$). Thus if $\co$ is the identity relation,
then $xy = yx$ for all $x,y\in S$, that is, $S$ is commutative. By (1), $S\cong S/\!\!\co$
is cancellative.
\end{proof}

Observe that in left zero semigroups (those satisfying the identity $xy=x$),
$\co$ is the universal relation, thus a congruence, but the semigroup is not
commutative.

In commutative semigroups, $p$-conjugacy is the identity, and nontrivial
cancellative semigroups cannot have a zero. Thus the following result holds.

\begin{cor}
\label{cide}
Let $S$ be a commutative and cancellative semigroup. Then
$\cp$, $\co$, and $\con$ all coincide, and are equal to the identity relation.
\end{cor}

By the definition of the notion of conjugacy, all semigroup conjugacies coincide in a group.
The following result is a sort of converse.
\begin{cor}
\label{cide2}
Let $S$ be an epigroup. Then $\cp$, $\co$, $\ctr$ and $\con$ all coincide
and are equal to the identity relation if and only if $S$ is a commutative group.
\end{cor}
\begin{proof}
Regarding the direct implication, observe that if $\ctr$ is the identity, then the semigroup
is completely regular and commutative; in addition $\co$ trivial implies that $S$ is cancellative.
It is well known that a regular cancellative semigroup is a group.

The converse is obvious.
\end{proof}

\medskip

Now we discuss conditions under which our various notions of conjugacy on a
semigroup $S$ coincide with the universal relation $S\times S$. Regarding $o$-conjugacy,
no characterization seems likely, because of what we have noted multiple times already:
$\co$ is universal in any semigroup with a zero.

Thus we pass immediately to trace conjugacy in epigroups. One would guess that in
epigroups with universal trace conjugacy, each subgroup is trivial, and this does indeed
turn out to be the case; see part (2) of the following result. More interestingly, the theorem
shows that the class of epigroups in which trace conjugacy is universal forms a variety.

\begin{theorem}
\label{thm:tr_universal}
Let $S$ be an epigroup. The following are equivalent:
\begin{enumerate}[label=\textup{(\arabic*)}]
    \item $\ctr$ is the universal relation;
    \item $E(S)$ is an antichain and for all $x\in S$, $x'' = x^{\omega}$;
    \item for all $x,y\in S$, $x'yx' = x'$;
    \item for all $x,y\in S$, $x^{\omega} y x^{\omega} = x^{\omega}$;
    \item for all $x\in S$, $e\in E(S)$, $exe=e$.
\end{enumerate}
\end{theorem}
\begin{proof}
We prove (1)$\implies$(2)$\implies$(3)$\implies$(1) and (3)$\iff$(4)$\iff$(5).

Assume (1), that is, $\ctr\,=\,S\times S$.  Since $\ctr\,\subseteq\,\co$, it follows that
$\ctr\,=\,\co$. By Proposition \ref{lem:epi-oc}, $E(S)$ is an antichain.
Now fix an idempotent $e$. Since $\ctr$ is universal, $e\,\ctr\,x'$ for all $x\in S$.
Thus by Theorem \ref{thm:trace}(6), there exist $u,v\in S^1$ such that $e = e'' = uv$ and
$x' = (x')'' = vu$. Now
\begin{equation}
\label{eqn:tr_univ_1}
x^{\omega} = (x')^{\omega} = x'x'' = vu(vu)' \by{eq:epi4} v(uv)'u = ve'u = veu\,.
\end{equation}
Thus
\[
x'' = x' x^{\omega} \by{eqn:tr_univ_1} vuveu = veeu = veu = x^{\omega}\,.
\]
This establishes (2).

\smallskip

Assume (2) holds. Note that for all $x\in S$, $x' = x''' = (x^{\omega})' = x^{\omega}$, so $x'$ is
idempotent. We show that for all $x,y\in S$, $x'(yx')'$ is idempotent, freely using \eqref{eq:epi4}
to rewrite this as $(x'y)'x'$:
\[
\underbrace{x'(yx')'}x'(yx')' = (x'y)'\underbrace{x'x'}(yx')'
= \underbrace{(x'y)'x'}(yx')' = x'\underbrace{(yx')'(yx')'} = x'(yx')'\,.
\]
Next we show that $x'\leq x'(yx')'$: $x'\cdot x'(yx')' = x'(yx')'$ and
$x'(yx')'x' = (x'y)'x'x' = (x'y)'x' = x'(yx')'$.

Now since $E(S)$ is an antichain and $x', x'(yx')'\in E(S)$, we conclude that $x'(yx')' = x'$.
Finally, we have $(yx')' = (yx')^{\omega} = y\underbrace{x'(yx')'} = yx'$. Therefore $x'yx' =x'(yx')'= x'$,
which establishes (3).

\smallskip

Assume (3) holds. For $x,y\in S$, set $u = x''y''$ and $v = y''x''$. Then
$x'' = x'' y'' y'' x'' = uv$ and $y'' = y'' x'' x'' y'' = vu$.
Thus $x''\,\cp\,y''$, and so $x\,\ctr\,y$ by Theorem \ref{thm:trace}(6). Thus $\ctr$ is the
universal relation, that is, (1) holds.

\smallskip

Next, once again assume (3) holds.
Taking $y = x''$, we have $x^{\omega} y x^{\omega} = x(x'(yx)x') = xx'  = x^{\omega}$, so that (4) holds.

\smallskip

Assume (4) holds. Taking $y = x$, we obtain $x^{\omega} = x^{\omega} x x^{\omega} = x'' x^{\omega} = x''$.
Thus $xx' = x''$, and so $x' = x'xx' = x'x'' = (x')^{\omega} = x^{\omega}$. Therefore
$x'yx' = x^{\omega}yx^{\omega} = x^{\omega} = x'$. This establishes (3).

\smallskip

Finally, the equivalence of (4) and (5) is obvious.
\end{proof}

\medskip

Next we discuss semigroups in which $p$-conjugacy is the universal relation. Our description is
complete for semigroups with idempotents, and partial for semigroups without idempotents.

First we need a definition. A \emph{rectangular band} is an idempotent semigroup satisfying
the identity $xyx=x$ for all $x,y$. Every rectangular band is completely simple, and in fact,
is isomorphic to the Rees matrix semigroup $I\times G\times\Lambda$ with $G=\{1\}$
\cite[Thm.~1.1.3]{Ho95}.

\begin{theorem}
\label{tpun}
Let $S$ be a semigroup.
\begin{enumerate}[label=\textup{(\arabic*)}]
 \item If $S$ is a rectangular band, then $\cp$ is the universal relation.
 \item If $\cp$ is the universal relation in $S$, then $S$ is simple.
 If, in addition, $S$ contains an idempotent, then $S$ is a rectangular band.
\end{enumerate}
\end{theorem}
\begin{proof}
(1) Let $S$ be a rectangular band. For $x,y\in S$, set $u = xy$, $v= yx$.
Then $x = xyx = xyyx = uv$ and $y = yxy = yxxy = vu$. Therefore $x\,\cp\,y$
for all $x,y\in S$, that is, $\cp$ is universal.

(2) Let $S$ be a semigroup in which $\cp$ is the universal relation. We first show
that $S$ is simple. For $a\in S$, $S^1a S^1$ is the principal ideal of $S$ generated by $a$.
We want to show that $S^1a S^1=S$. Let $b\in S$. If $b=a$, then clearly $b\in S^1a S^1$
Suppose that $b\ne a$. Since $a\cp b$ and $a\ne b$,
there exist $u_1, v_1\in S$ such that $a = u_1v_1$, $b = v_1u_1$.
Note that $u_1\neq v_1$ (since otherwise $a=b$), so there exist $u_2,v_2\in S$
such that $u_1 = u_2v_2$, $v_1 = v_2u_2$. Now, if $u_2 = a$, then
$b = v_1u_1 = v_2au_1 \in SaS$, so we may assume $u_2 \neq a$. Then, there
exist $u_3,v_3\in S$ such that $a = u_3v_3$, $u_2 = v_3u_3$, and so
\[
b = v_1 u_1 = v_2u_2u_2v_2= v_2v_3u_3v_3u_3v_2 = v_2v_3au_3v_2 \in SaS\,.
\]
Hence $S^1a S^1=S$, and so $S$ is simple.

Now suppose $S$ has an idempotent $e$. We will show that $S$ satisfies
the identity $x^3 = x^2$. Since $x\cp e$, there exist $u,v\in S^1$
with $x = uv$, $e = vu$. Then
\[
xxx = uvuvuv = ueev = uev= uvuv = xx\,.
\]

The identity $x^3 = x^2$ implies that $S$ is an epigroup in $\mathcal{E}_2$ with $x' = x^2$,
that is, $x'xx' = x^5 = x^2 = x'$, $xx'=x'x$ and $x^3 x' = x^2$. Since $\cp\,\subseteq\,\ctr$,
we have that $\ctr$ is the universal relation. By Theorem~\ref{thm:tr_universal}(2),
$E(S)$ is a chain, that is, every idempotent is primitive.

We have now shown that $S$ is completely simple. In particular, $S$ is completely regular
and the epigroup pseudo-inverse $x' = x^2$ is actually an inverse. Thus
$x = xx'x = x^4$. But this together with $x^3 = x^2$ imply $x^2 = x$ for all $x\in S$,
that is, $S$ is an idempotent semigroup. Now using Theorem~\ref{thm:tr_universal}(3),
we conclude that $xyx = x'yx' = x$ for all $x,y\in S$, that is, $S$ is a
rectangular band.
\end{proof}




\begin{example}
By Theorem~\ref{tpun}, if $\cp$ is universal in a semigroup $S$, then $S$ is
simple.
If $S$ does not have an idempotent, then the converse is not necessarily true.
Let $X$ be a countably infinite set. Denote by $\Gamma(X)$ the semigroup of all
injective
mappings from $X$ to $X$. For $\al\in\Gamma(X)$, let $\ima(\al)$ denote the
image of $\al$.
The set $S$ consisting of all $\al\in\Gamma(X)$ such that $X\setminus\ima(\al)$
is infinite
is a subsemigroup of $\Gamma(X)$, called the Baer-Levi semigroup
\cite[\S8.1]{ClPr64}. The Baer-Levi
semigroup is simple without idempotents \cite[Theorem~8.2]{ClPr64}. Partition
the set $X$ as follows:
\[
X=\{x,y\}\cup\{z^1_1,z^1_2,\ldots\}\cup\{z^2_1,z^2_2,\ldots\}\cup\{z^3_1,z^3_2,
\ldots\}\cup\ldots.
\]
Define $\al,\bt\in S$ by:
\[
\al(x)=x,\,\,\al(y)=y,\,\,\al(z^j_i)=z^j_{i+1},\,\,\bt(x)=y,\,\,\bt(y)=x,\,\,
\bt(z^j_i)=z^j_{i+1}.
\]
Then $(\al,\bt)\not\in\,\,\cp$ by \cite[Proposition~4]{KuMa07}, so $\cp$ is not
the universal relation.
\end{example}

Since, for example, every finite semigroup has an idempotent,
Theorem~\ref{tpun} implies an immediate corollary.

\begin{cor}
In a finite semigroup (or more generally, an epigroup) $S$, $\cp$ is the universal relation if and only if
$S$ is a rectangular band.
\end{cor}

We conclude this section with some results that extend to semigroups familiar
results on conjugacy in groups.

For elements $a_1,a_2,b_1,b_2$ in a group, if $a_1 a_2$ is conjugate to $b_1
b_2$, then
$a_2 a_1$ is conjugate to $b_2 b_1$. This result carries over to semigroups as
follows.
A semigroup $S$ with zero is \emph{categorical at zero} if for all $a,b,c\in
S$,
$ab\ne0$ and $bc\ne0$ imply $abc\ne0$ \cite[vol.~2, p.~73]{ClPr64}.
	
\begin{theorem}
Let $S$ be a semigroup.
\begin{enumerate}[label=\textup{(\arabic*)}]
 \item For all $a_1,a_2,b_1,b_2\in S$,  $a_1a_2\co b_1b_2$
implies $a_2a_1\co b_2b_1$.
 \item If $S$ is categorical at zero and
$a_1a_2,a_2a_1,b_1b_2,b_2b_1\neq 0$, then $a_1a_2\con b_1b_2$ implies
$a_2a_1\con b_2b_1$.
 \item The following statements are equivalent:
	\begin{enumerate}[label=\textup{(\alph*)}]
	\item $\cp$ is transitive in $S$;
	\item For all $a_1,a_2,b_1,b_2\in S$, $a_1a_2\cp b_1b_2$ implies $a_2a_1\cp b_2b_1$.
	\end{enumerate}
 \item For all $a_1,a_2,b_1,b_2\in S$ such that $a_1a_2,b_1b_2,a_2a_1,b_2b_1\in\Epi(S)$,
$a_1a_2\ctr b_1b_2$ implies $a_2a_1\ctr b_2b_1$.
\end{enumerate}
\end{theorem}
\begin{proof}
Let $a_1a_2\co b_1b_2$. This implies that, for some $c,d\in S$, $a_1a_2c=cb_1b_2$ and $b_1b_2d=da_1a_2$.
Then
\begin{equation}
\label{jan2}
a_2a_1(a_2cb_1)=a_2(a_1a_2c)b_1=a_2(cb_1b_2)b_1\,\,\mbox{ and }\,\,b_2b_1(b_2da_1)=(b_2da_1)a_2a_1.
\end{equation}
Thus $a_2a_1\co b_2b_1$. We have proved (1).

Regarding $\con$, suppose that $S$ is categorical at zero and let $a_1a_2,a_2a_1,b_1b_2,b_2b_1\neq 0$.
Suppose that $a_1a_2\con b_1b_2$.
This implies that  $a_1a_2c=cb_1b_2$ and $b_1b_2d=da_1a_2$ for some
$c\in \pp^1(a_1a_2)$ and $d\in \pp^1(b_1b_2)$. As in the proof of (1), we
obtain equalities \eqref{jan2}.
It remains to prove that $a_2cb_1\in\pp^1(a_2a_1)$ and $b_2da_1\in\pp^1(b_2b_1)$.
First we observe that in any semigroup categorical at zero, $x\in \pp^1(y)$ if
and only if $yx\neq 0$.
Since $c\in \pp^1(a_1a_2)$, it follows that $cb_1b_2=a_1a_2c\neq 0$, and hence
$a_2c\neq 0\neq cb_1$. Thus $a_2cb_1\ne0$
since $S$ is categorical at zero. Similarly, since $a_2a_1\ne0$ and
$a_1a_2\ne0$, we have $a_2a_1a_2\ne0$.
Now $a_2a_1a_2\ne0$ and $a_2cb_1\ne0$ imply $a_2a_1a_2cb_1\ne0$, which implies
that $a_2cb_1\in\pp^1(a_2a_1)$.
Similarly, $b_2da_1\in\pp^1(b_2b_1)$, which concludes the proof of (2).

Regarding $\cp$, we start by proving $\textup{(a)}\implies\textup{(b)}$. Suppose $\cp$
is transitive
and let $a_1a_2\cp b_1b_2$. By the definition of $\cp$, we have $xy \cp yx$ for
all $x,y\in S$. Thus
\[
a_2a_1\cp a_1a_2 \cp b_1b_2 \cp b_2b_1,
\]
which implies $a_2a_1\cp b_2b_1$ since $\cp$ is transitive.

For $\textup{(b)}\implies\textup{(a)}$, assume that
$a_1a_2\cp b_1b_2$ implies $a_2a_1\cp b_2b_1$ for all $a_1,a_2,b_1,b_2\in S$.
Let $a,b,c\in S$ and suppose $a\cp b$ and $b\cp c$.
Then $a=xy$, $b=yx=uv$, and $c=vu$ for some $x,y,u,v\in S^1$.
Thus $yx\cp uv$ (since $xy=uv=b$), and hence $xy\cp vu$ (by the hypothesis),
that is, $a\cp c$. Therefore, $\cp$ is transitive.

Finally, the result for $\ctr$ follows from Theorem~\ref{thm:steinbergExtensionEpigroups}(3).
\end{proof}

In a group, if $a$ and $b$ are conjugate, then $a^k$ and $b^k$ are also
conjugate for all positive integers $k$.
This fact generalizes to the conjugacies $\cp$, $\con$, and $\co$ in semigroups.

\begin{theorem}
\label{thm:powers}
Let $S$ be a semigroup and let $\sim\,\,\in\{\co,\con,\cp\}$. Then for all
$a,b\in S$ and integers $k\geq1$,
$a\sim b$ implies $a^k\sim b^k$.
\end{theorem}
\begin{proof}
Let $a,b\in S$ and $c\in S^1$ be such that $ac=cb$. We claim that $a^kc=cb^k$
for all integers $k\geq1$.
We proceed by induction on $k$. The claim is certainly true for $k=1$. Let
$k\geq1$ and suppose that $a^kc=cb^k$.
Then
$a^{k+1}c=a(a^kc)=a(cb^k)=(ac)b^k=cb^{k+1}$. The claim has been proved. The
result follows immediately for $\co$ and $\con$.

For $\cp$, the desired result is \cite[Lem.~2]{ganna}: if, say, $a=cd$ and $b=dc$, then $a^k=((cd)^{k-1}c)d$ while
$b^{k}=d ((cd)^{k-1}c)$.
\end{proof}

The same result is true for trace conjugacy and epigroup elements.

\begin{theorem}
\label{tpow}
Let $S$ be a semigroup. Then for all $a,b\in\Epi(S)$ and integers $k\geq1$,
$a\ctr b$ implies $a^k\ctr b^k$.
\end{theorem}
\begin{proof}
Suppose that $a\ctr b$. Then $a''\cp b''$ by Theorem~\ref{thm:trace}, and so $(a'')^k\cp (b'')^k$ by Theorem~\ref{thm:powers}.
Since $(a'')^k=(a^k)''$ and $(b'')^k=(b^k)''$, we have $(a^k)''\cp(b^k)''$, and so $a^k\ctr b^k$ by Theorem~\ref{thm:trace}.
\end{proof}

In a group, if $a$ and $b$ are conjugate, then $a\inv$ and $b\inv$ are also conjugate.
This fact generalizes to $o$-conjugacy and $p$-conjugacy in epigroups.
(See Proposition~\ref{pctrainv} for a stronger result for trace conjugacy.)

\begin{theorem}
\label{thm:op-primes}
Let $S$ be an epigroup and let $\sim\,\,\in\{\co,\cp\}$. Then for all $a,b\in
S$, $a\sim b$ implies $a'\sim b'$.
\end{theorem}
\begin{proof}
Suppose $a\co b$, so $ac=cb$ and $da=bd$ for some $c,d\in S^1$. Set $g =
aa'cb'$
and $h = bb'da'$.
Then
\[
a'g = \underbrace{a'aa'}cb' \by{eq:epi1} a'c\underbrace{b'} \by{eq:epi1}
a'c\underbrace{b'b}b' \by{eq:epi2} a'\underbrace{cb}b'b'
= \underbrace{a'a}cb'b' \by{eq:epi2} aa'cb'b' = gb'\,,
\]
and an almost identical calculation shows $b'h = ha'$. Thus $a'\co b'$.

Now suppose $a\cp b$. Then $a=cd$ and $b=dc$ for some $c,d\in S^1$. Set $u =
c$,
$v = d(cd)'(cd)'$. Then
$uv = cd(cd)'(cd)' = (cd)'cd(cd)' = (cd)' = a'$, using \eqref{eq:epi2} and
\eqref{eq:epi1},
and
$vu = d(cd)'(cd)'c = (dc)'dc(dc)' = (dc)' = b'$, using \eqref{eq:epi4} twice
followed by \eqref{eq:epi1}.
Thus $a'\cp b'$.
\end{proof}

In a group, if $a$ and $b$  are conjugate and $a^m=a^k$ for some integers
$m,k\geq1$, then $b^m=b^k$.
This result does not hold in general for semigroups,
but we have the following for $\cp$.

\begin{theorem}
\label{thm:finite_exp}
Let $S$ be a semigroup and let $a,b\in S$ such that $b$ is
an epigroup element with $b^t$ ($t\geq 1$) lying in a subgroup of
$S$. If $a\cp b$ and $a^m=a^k$
for some integers $m,k\geq t$,  then $b^m = b^k$.
\end{theorem}
\begin{proof}
Since $a\cp b$, $a = cd$ and $b = dc$ for some $c,d\in S^1$. Since $b^t$
is in a subgroup of $S$, we have, by \eqref{etfh}, $b^{n+1}b'=b^n$ for every integer $n\geq t$.
Thus
\[
b^m = b^{m+1} b' = d(cd)^m c b' = d a^m c b' = d a^k c b' = (dc)^{k+1} b' =
b^{k+1} b' = b^k,
\]
which completes the proof.
\end{proof}

\begin{cor}
\label{jan3}
Let $S$ be an epigroup in $\mathcal{W}$. If $a,b\in S$ satisfy $a\cp b$ and
$a^m=a^k$ for some integers $m,k\geq1$,
then $b^m = b^k$.
\end{cor}
\begin{proof}
Since $a\cp b$, we have $a = cd$ and $b = dc$ for some $c,d\in S^1$. Since $b''
= (dc)'' = dc = b$ by \eqref{eq:xy''}, $b$ is completely regular, so
Theorem~\ref{thm:finite_exp} applies with $t=1$.
\end{proof}

Theorem~\ref{thm:finite_exp} fails for $\co$. Indeed, if $S$ has a zero
as its unique idempotent,
then $\co$ is the universal relation, but $0^2=0$ while $a^2\neq a$ for every
nonzero $a\in S$.

\section{Open problems}
\label{spro}
\setcounter{equation}{0}

We conclude this paper with some natural questions related to conjugacy.

In \S\ref{ssym}, we characterized $c$-conjugacy in the symmetric inverse semigroup
$\mi(X)$
for a countable set~$X$. Descriptions of $\cp$ in this semigroup can be
found in \cite{GaKo93}
and \cite{KuMa07}.

\begin{prob}
Characterize the relations $\con$ and $\cp$ in $\mi(X)$ for an uncountable set
$X$.
\end{prob}

A characterization of $c$-conjugacy in the \emph{full transformation semigroup}
$T(X)$ on any set $X$ was obtained in \cite{AKM14}. For a finite set $X$,
$p$-conjugacy in $T(X)$ was described
in \cite{KuMa07}. The \emph{partition semigroup} $\mathcal{P}_{X}$ on a set~$X$
\cite{Eas11,EF12}
has both $T(X)$ and the symmetric inverse semigroup $\mi(X)$ as subsemigroups.

\begin{prob}
Characterize the relations $\con$ and $\cp$ in $\mathcal{P}_{X}$, and $\ctr$ restricted to the epigroup elements.
\end{prob}

We proved in \S\ref{sec:epigroups} that $p$-conjugacy is transitive in
completely regular semigroups
and their variants, but noted that the epigroup variety $\mathcal{W}$ does not
include all epigroups in which
$\cp$ is transitive.

\begin{prob}
Find other classes of semigroups in which $p$-conjugacy is transitive. Describe
the [$E$-unitary] inverse semigroups in which
$p$-conjugacy is transitive. Ultimately, classify the class of semigroups in
which $\cp$ is transitive.
\end{prob}

As already noted, $\cp$ is transitive in free semigroups. Free semigroups are
both cancellative and embeddable in groups.

\begin{prob}
Is $\cp$  transitive in every cancellative semigroup? In every semigroup
embeddable in a group?
\end{prob}

In this paper, we studied conjugacy in the symmetric inverse semigroup
$\mi(X)$,
but
many other transformation semigroups, or endomorphism monoids of some
relational
algebras,
may be considered.

\begin{prob}
For $\con$, $\cp$, and $\ctr$, characterize the conjugacy classes  and
calculate their number  for other
transformation semigroups such as, for example, those appearing in the problem
list of \cite[Section~6]{ak} or those
appearing in the large list of transformation semigroups included in
\cite{vhf}.
 Especially interesting would be a
characterization of the conjugacy classes in the centralizers of idempotents
\cite{arko2,arko1}, or in semigroups
whose group of units has an especially rich structure
\cite{arbemics,arcameron,arcameron2,arsc}.
\end{prob}

The classes described in the preceding problem have linear analogs and hence
can
be extended to the more
general setting of independence algebras.

\begin{prob}
Characterize $\con$, $\cp$, and $\ctr$ in the endomorphism monoid of an independence algebra.
In \cite{abk}, a problem on independence algebras was solved using their
classification theorem; it is reasonable to guess that the same
technique can be used to solve the problem proposed here. (For historical notes
on how a problem on idempotent generated
semigroups \cite{1,2} led to these algebras, see \cite{aeg,arfo}; for
definitions and basic results, see
\cite{araujo29,araujo2,araujo,armi2,aw,cameron,fou1,fou2,gould}.)

Similarly interesting would be the characterisation of the conjugacy classes
for the endomorphism monoids of free objects \cite{armi3}
or for the endomorphisms of algebras admitting some general notion of
independence \cite{aw}. Regarding the latter, we propose the
problem of calculating the conjugacy classes in the endomorphisms  of
$MC$-algebras, $MS$-algebras, $SC$-algebras, and $SC$-ranked
algebras \cite[Chapter 8]{aw}. A first step would be to solve the conjugacy
problem for the endomorphism monoid of an $SC$-ranked
free $M$-act \cite[Chapter 9]{aw}, and for an $SC$-ranked free module over an
$\aleph_1$-Noetherian ring \cite[Chapter 10]{aw}.

Since all varieties of bands are known, especially interesting would be the
description of the conjugacy classes of the
endomorphism monoid of the free objects of each variety of bands (for details
and references, see \cite{bands}).
\end{prob}

The study of the intersection of $\con$ with other conjugacies was
omitted from this paper. This suggests the following problem.

\begin{prob}
Let $\sim\,\,\in\{\co,\cp,\ctr\}$.  Study the notion of conjugacy $\con\cap \sim$.
In particular, describe it in the various
types of transformation semigroups listed in the previous problems.
\end{prob}

We have proved that if a semigroup $S$ has an idempotent, then $\cp$ is the
universal relation in $S$ if and only if
$S$ is a rectangular band. We have also proved that every semigroup in which
$\cp$ is universal is simple, and noted
that there are simple semigroups without idempotents in which $\cp$ is not
universal.

\begin{prob}
Describe the simple semigroups without idempotents in which $p$-conjugacy is
the universal relation.
\end{prob}

We know that $o$-conjugacy is universal in the semigroups with zero.

\begin{prob}
Describe the semigroups without zero in which $o$-conjugacy (and thus
$c$-conjugacy) is the universal relation.
\end{prob}

We will say that a given conjugacy $\sim$ is \emph{partition covering}
if for every set $X$ and for every partition $\tau$ of $X$,
there exists a semigroup $S$ with universe $X$ such that the $\sim$-conjugacy
classes on $S$ form
the same partition as $\tau$.

\begin{prob}
Is it true that $o$-conjugacy [$p$-conjugacy, $\ctr$-conjugacy] is a partition-covering relation?
\end{prob}

We have used the GAP package \emph{Smallsemi} \cite{Smallsemi} to check that
this is true for all
$X=\{1,\ldots,n\}$ where  $1\leq n\leq 6$, and $\co$ or $\cp$.
As \emph{Smallsemi} contains all semigroups up to order $8$, the following
special case of the preceding
problem might take a long time to compute, but it is certainly computationally
feasible.

\begin{prob}
Is it true that $o$-conjugacy [$p$-conjugacy, trace conjugacy] is a partition-covering relation
for all sets of size at most $8$? What about $9$?
\end{prob}

In Theorem \ref{thm:strong-o}, we showed that $o$-conjugacy in epigroups is
equivalent to a stronger
notion of conjugacy. Call elements $a,b$ of a semigroup $S$ \emph{strongly}
$o$-\emph{conjugate},
denoted by $a\sim_{so} b$, if there exist mutually inverse $g,h\in S^1$ such
that $ag = gb$
and $bh = ha$. The relation $\sim_{so}$ is evidently reflexive and symmetric,
and
$\sim_{so}\,\,\subseteq\,\,\sim_{o}$. Theorem \ref{thm:strong-o} can be restated
as saying that in epigroups, $\sim_{so}\,\,=\,\,\co$. This result is not
true
in general.
For example, the transformations $\al$ and $\bt$ defined in the proof of
Theorem~\ref{tgrd}
are $o$-conjugate but not strongly $o$-conjugate in the semigroup $\gx$.

\begin{prob}
Find natural classes of semigroups in which $\sim_{so}\,\,=\,\,\co$.
\end{prob}

Since $\co$ is transitive in arbitrary semigroups, Theorem
\ref{thm:strong-o}
implies that $\sim_{so}$ is
transitive in epigroups. It is also easy to see that $\sim_{so}$ is transitive
in inverse
semigroups. (If $a\sim_{so} b\sim_{so} c$, then $ag=gb$, $bg\inv = g\inv a$,
$bk
= kc$,
$ck\inv = k\inv b$ for some $g,k$. Thus $agk = gbk=gkc$ and $c(gk)\inv = ck\inv
g\inv
= k\inv bg\inv = k\inv g\inv a = (gk)\inv a$.)

\begin{prob}
Is $\sim_{so}$ transitive in arbitrary semigroups? In regular
semigroups?
\end{prob}

The analog of strong $o$-conjugacy for $\con$ is as follows: Call elements
$a,b$ of a
semigroup $S$ \emph{strongly} $c$-\emph{conjugate}, denoted by $a\sim_{sc} b$,
if there exist
$g\in \mathbb{P}^1(a)$, $h\in \mathbb{P}^1(b)$ such that $g,h$ are mutually
inverse and
$ag=gb$, $bh=ha$. Evidently $\sim_{sc}\,\,\subseteq\,\,\con$. Theorem
\ref{thm:strong-c}
can be rephrased as saying that for epigroups in $\mathcal{W}$, $\con\ = \ \sim_{sc}$.

\begin{prob}
Does Theorem \ref{thm:strong-c} generalize to all epigroups?
Does there exist a semigroup with a pair of $c$-conjugate elements which are
not
strongly
$c$-conjugate? A regular such semigroup? An inverse semigroup?
\end{prob}

{
\begin{prob}
 Is it possible to prove a result similar to Theorem~\ref{thm:c0s}, replacing
regular epigroups by epigroups in $\mathcal{W}$?
For semigroups without zero we
have a similar result. Possibly, it is necessary to start by proving that $x\con x''$ for all $x$ such that
$x''\neq 0$. If such result could be proved, then the result would follow as in the case without zero.
\end{prob}
}

{
\begin{prob}
Is there an example of a semigroup $S$ in which $\co$ is a congruence, but
$S/\!\!\!\co$ is not cancellative?
\end{prob}
}

The coordinatization theorem (\cite[Definition A.4.18]{rs}) for rectangular bands is probably the most basic such result involving  two of Green's relations.

\begin{prob}
Find a class of semigroups admitting a coordinatization theorem in terms of $\con$ and $\ctr$ [respectively, $\con$ and $\cp^*$]. In particular, classify the semigroups in which $\con\cap \ctr$ [respectively, $\con\cap\cp^*$] is the identity relation.
\end{prob}

The class ${\mathcal W}$ seems a very interesting generalization of the class of completely regular semigroups. It is likely that many of the results 
for the latter carry over to the former.

\begin{prob}
Generalize for ${\mathcal W}$ the main results on completely regular semigroups. In particular, is it true that $\cp$ is transitive in the variants of $\mathcal{W}$?
\end{prob}

Consider the variety $\mathcal{V}$ of unary semigroups
$(S,\cdot,{}')$ defined by associativity, $x'xx'=x'$, $xx'=x'x$ and
\begin{align}
x'' y &= xy\,,    \label{eqn:nv3} \\
x y'' &= xy\,.    \label{eqn:nv4}
\end{align}

This class also generalizes completely regular semigroups and appears to be as interesting as ${\mathcal W}$.

\begin{prob}
Generalize for ${\mathcal V}$ the main results on completely regular semigroups. In particular, is it true that $\cp$ is transitive in the variants of $\mathcal{V}$?
\end{prob}

In \cite{Boyd_et_al} there are two generalizations of the notion of variants of
semigroups; one appears in Proposition~2.1
and relies on translations, and the other is provided by the concept of
\emph{interassociates}
(for definitions we refer the reader to \cite{Boyd_et_al}).

\begin{prob}
Do the results on variants in this paper carry over to the two generalizations
introduced in~\cite{Boyd_et_al}?
\end{prob}

As seen in Figure \ref{fig0}, $\con$ is not related to $\cp$ or $\ctr$. 

\begin{prob}
Is it possible to find an infinite set of notions of conjugacy for semigroups, first order definable, and that form an anti-chain [infinite chain]?
\end{prob}

The final problem deals with the converse of Example \ref{4.14}.

\begin{prob}
Is it true that if $\cp$ is transitive in all variants of a semigroup, then it
is also transitive in the semigroup itself?
\end{prob}

\section{Acknowledgments}
The authors would like to thank the referee for the excellent suggestions that led to a much improved paper. 


\end{document}